\documentclass[preprint, 12pt]{amsart}
\usepackage{srcltx}
\usepackage{eurosym}
\usepackage{mathtools}
\usepackage{amsmath}
\usepackage{amsfonts}
\usepackage{amssymb}
\usepackage{amsthm}
\usepackage{graphicx}
\usepackage{mathrsfs}
\usepackage{xcolor}
\usepackage{exscale}
\usepackage{latexsym}
\usepackage{esvect}
\usepackage{euler}
\usepackage[overload]{empheq}
\renewcommand\emph[1]{\color{auburn}{{#1}}}
\definecolor{airforceblue}{rgb}{0.36, 0.54, 0.66}
\definecolor{auburn}{rgb}{0.43, 0.21, 0.1}
\usepackage[colorlinks,plainpages=true,pdfpagelabels,hypertexnames=true,colorlinks=true,pdfstartview=FitV,linkcolor=blue,citecolor=red,urlcolor=black,backref=page]{hyperref}
\PassOptionsToPackage{unicode}{hyperref}
\PassOptionsToPackage{naturalnames}{hyperref}
\definecolor{alizarin}{rgb}{0.82, 0.1, 0.26}

\usepackage{enumerate}
\usepackage[shortlabels]{enumitem}
\usepackage{bookmark}
\usepackage{wasysym}
\usepackage{esint}
\usepackage[titletoc,toc,page]{appendix}
\usepackage[ddmmyyyy]{datetime}
\numberwithin{equation}{section}
\everymath{\displaystyle}
\usepackage[margin=2.2cm]{geometry}

\usepackage[capitalize,nameinlink]{cleveref}
\crefname{section}{Section}{Sections}
\crefname{subsection}{Subsection}{Subsections}
\crefname{condition}{Condition}{Conditions}
\crefname{hypothesis}{Hypothesis}{Hypothesis}
\crefname{assumption}{Assumption}{Assumptions}
\crefname{lemma}{Lemma}{Lemmas}
\crefname{claim}{Claim}{Claims}
\crefname{remark}{Remark}{Remarks}

\crefformat{equation}{\textup{#2(#1)#3}}
\crefrangeformat{equation}{\textup{#3(#1)#4--#5(#2)#6}}
\crefmultiformat{equation}{\textup{#2(#1)#3}}{ and \textup{#2(#1)#3}}
{, \textup{#2(#1)#3}}{, and \textup{#2(#1)#3}}
\crefrangemultiformat{equation}{\textup{#3(#1)#4--#5(#2)#6}}%
{ and \textup{#3(#1)#4--#5(#2)#6}}{, \textup{#3(#1)#4--#5(#2)#6}}%
{, and \textup{#3(#1)#4--#5(#2)#6}}

\Crefformat{equation}{#2Equation~\textup{(#1)}#3}
\Crefrangeformat{equation}{Equations~\textup{#3(#1)#4--#5(#2)#6}}
\Crefmultiformat{equation}{Equations~\textup{#2(#1)#3}}{ and \textup{#2(#1)#3}}
{, \textup{#2(#1)#3}}{, and \textup{#2(#1)#3}}
\Crefrangemultiformat{equation}{Equations~\textup{#3(#1)#4--#5(#2)#6}}%
{ and \textup{#3(#1)#4--#5(#2)#6}}{, \textup{#3(#1)#4--#5(#2)#6}}%
{, and \textup{#3(#1)#4--#5(#2)#6}}

\crefdefaultlabelformat{#2\textup{#1}#3}
\newtheorem{theorem}{Theorem}[section]
\newtheorem{lemma}[theorem]{Lemma}

\newtheorem{remark}[theorem]{Remark}        
  
\numberwithin{equation}{section}

\def\YYint#1#2#3{{\setbox0=\hbox{$#1{#2#3}{\iint}$}
\vcenter{\hbox{$#2#3$}}\kern-.50\wd0}}

\def\XXint#1#2#3{{\setbox0=\hbox{$#1{#2#3}{\int}$}
\vcenter{\hbox{$#2#3$}}\kern-.50\wd0}}

\makeatletter
\def\namedlabel#1#2{\begingroup
\def\@currentlabel{#2}%
\label{#1}\endgroup
}
\makeatother
\makeatletter
\newcommand{\rmh}[1]{\mathpalette{\raisem@th{#1}}}
\newcommand{\raisem@th}[3]{\hspace*{-1pt}\raisebox{#1}{$#2#3$}}
\makeatother

\newcommand{\descref}[2]{\hyperref[#1]{\textup{\textcolor{black}{(}\textcolor{blue}{\bf #2}\textcolor{black}{)}}}}

\makeatletter
\g@addto@macro\normalsize{%
\setlength\abovedisplayskip{3pt}
\setlength\belowdisplayskip{3pt}
\setlength\abovedisplayshortskip{1pt}
\setlength\belowdisplayshortskip{3pt}
}
\makeatletter
\def\ps@pprintTitle{%
\let\@oddhead\@empty
\let\@evenhead\@empty
\def\@oddfoot{}%
\let\@evenfoot\@oddfoot}
\makeatother
\usepackage{setspace}
\newcommand{\xeps}{\left(\frac{x}{\varepsilon}\right)}

\newcommand\RR{\mathbb{R}}

\newcommand{\eps}{\varepsilon}

\DeclareMathOperator{\dv}{div}

\DeclareMathOperator{\dist}{dist}

\DeclareMathOperator{\supp}{supp}

\allowdisplaybreaks

\definecolor{Plum}{HTML}{89b02e}

\definecolor{Violet}{HTML}{58429B}

\definecolor{OliveGreen}{HTML}{0d8795}

\usepackage[titletoc]{appendix}

\makeatletter
\def\ps@pprintTitle{%
\let\@oddhead\@empty
\let\@evenhead\@empty
\def\@oddfoot{}%
\let\@evenfoot\@oddfoot}
\makeatother

\begin{document}

\title[Parabolic equations with large drift and potential]{Quantitative periodic homogenization of\\ parabolic equations with large drift and potential}
\author[Kshitij Sinha]{}

\email{23d0781@iitb.ac.in}
\date{\today.}

\address{Department of Mathematics, Indian Institute of Technology Bombay, Mumbai, Maharastra, 400076, India}
\maketitle

\centerline{\scshape Kshitij Sinha}
\medskip
{\footnotesize
 \centerline{Department of Mathematics, Indian Institute of Technology Bombay, Mumbai, 400076, India}
}


\begin{abstract}
This work aims to study the rates in the context of periodic homogenization of parabolic problems with large lower order terms (both drift and potential). We demonstrate that the solution is a product of three terms: (i) a function of time, (ii) the ground-state of an exponential cell eigenvalue problem and (iii) the solution to a parabolic equation with zero effective drift. For the latter, we derive $\mathrm L^2$ rates in the homogenization limit.
\end{abstract}



\section{Introduction}
In this paper, for the unknown $u_\eps:\overline{\Omega}\times[0,T)\to\RR$, we study the following parabolic initial boundary value problem:
\begin{equation}\label{eq:main}
\left\{
\begin{aligned}
\frac{\partial u_{\eps}}{\partial t} - \dv\left(A\left(\frac{x}{\eps}\right) \nabla u_{\eps}\right) + \frac{1}{\eps}b\left(\frac{x}{\eps}\right) \cdot \nabla u_{\eps} + \frac{1}{\eps^2}c\left(\frac{x}{\eps}\right)u_{\eps} & = 0 \qquad \mbox{ in }(0,T) \times \Omega,\\
u_{\eps}(t,x) & = 0 \qquad \mbox{ on } (0,T)\times \partial \Omega,\\
u_{\eps}(0,x) & = u_{\eps}^0(x)\qquad \mbox{ in } \Omega.
 \end{aligned}
\right.
\end{equation}
Here the parameter $\eps\in(0,1)$. The bounded domain $\Omega\subset\RR^d$ is assumed to be open and with Lipschitz boundary $\partial\Omega$. The coefficients $A,b,c$ are assumed to be bounded in $\mathrm L^\infty$ and to be $Y$-periodic with the standard notation $Y:=[0,1)^d$ for the unit cell. Furthermore, $A$ is assumed to be uniformly elliptic, i.e.
\[
A(y)\xi\cdot\xi \ge \mu \left\vert \xi\right\vert^2 \quad \forall \xi\in\RR^d \, \mbox{ and for a.e. }y\in Y,
\]
for some $\mu>0$. We will be commenting more on the initial datum $u^0_\eps$ later. In the remainder of this manuscript, the standard notation $\Omega_T$ is used to denote the parabolic cylinder.\\
Our present work is inspired by 
\begin{itemize}
\item a result due to Gr\'egoire Allaire and Rafael Orive \cite{allaire2007homogenization} dealing with the parabolic problems with large drift and potential terms
\item an article of Gr\'egoire Allaire and Anne-Lise Raphael \cite{allaireraphael2007homogenization} which dealt with parabolic problems with large drift and potential terms in periodic porous media
\item a work of Yves Capdeboscq \cite{capdeboscq1998homogenization} addressing the periodic homogenization of eigenvalue problems for elliptic operators involving large drift terms
\end{itemize}
We also recall a work due to Gr\'egoire Allaire and co-authors \cite{allaire2004homogenization} dealing with parabolic problems with large potential terms (and no drift terms). Our present work is likewise related to the work of Patrizia Donato and Andrey Piatnitski \cite{donato2005averaging}. In this work, we demonstrate that the solution to the initial boundary value problem \eqref{eq:main} can be factorized as follows:
\begin{equation}\label{eq:factorization}
u_{\varepsilon}(t,x)=e^{\frac{-\lambda t}{\varepsilon^2}} v_{\varepsilon}(t,x)\psi\left(\frac{x}{\varepsilon}\right)
\end{equation}
where $(\lambda,\psi)$ are an eigenpair associated with a certain exponential $Y$-periodic eigenvalue problem. The use of exponential $Y$-periodic eigenvalue problems in homogenization can be traced back to \cite[page 383]{bensoussan2011asymptotic}. Allaire and Orive in \cite[Lemma 2.1]{allaire2007homogenization} (see also \cite[Lemma 3]{capdeboscq1998homogenization}, where Capdeboscq treats operators with no zeroth order terms) show the existence of a unique $\theta\in\mathbb{R}^d$ such that the eigenvalue problems (with $A^\ast$ denoting the adjoint of $A$):
\begin{equation}\label{eq:theta-direct-evp}
\left\{
\begin{aligned}
-\dv_y ({A}(y)\nabla \psi(y)) + b(y)\cdot\nabla_y\psi(y) + c(y) \psi(y) = \lambda \psi(y) \qquad \mbox{ in }(0,1)^d\\
y \mapsto \psi(y)e^{-2\pi \theta \cdot y}  \quad \mbox{ is }Y\mbox{-periodic},
\end{aligned}
\right.
\end{equation}
and
\begin{equation}\label{eq:theta-adjoint-evp}
\left\{
\begin{aligned}
-\dv_y ({A}^\ast(y)\nabla \psi^\ast(y)) - \dv_y(b(y)\psi^\ast(y)) + c(y) \psi^\ast(y) = \lambda^\ast \psi^\ast(y) \qquad \mbox{ in } (0,1)^d\\
y \mapsto \psi^\ast(y)e^{2\pi \theta \cdot y} \quad \mbox{ is }Y\mbox{-periodic},  
\end{aligned}
\right.
\end{equation}
have the following properties:
\begin{itemize}
\item Both the problems \eqref{eq:theta-direct-evp} and \eqref{eq:theta-adjoint-evp} have the same first eigenvalue $(\lambda=\lambda^\ast)$ which is real and simple.
\item The first eigenfunctions of \eqref{eq:theta-direct-evp} and \eqref{eq:theta-adjoint-evp} are bounded and can be chosen such that
\begin{equation}\label{eq:eigenfn-normalisation}
\min_Y\{\psi, \psi^\ast\} \ge a > 0
\qquad
\mbox{ and }
\qquad 
\int_Y \psi(y) \psi^\ast(y)\, {\rm d}y = 1.
\end{equation}
\item The function $\beta:Y\to\mathbb{R}^d$ defined by
\begin{equation}\label{eq:defn-beta}
\beta(y) := \psi(y)\psi^\ast(y)b(y) + \psi(y)A^\ast(y)\nabla_y \psi^\ast(y) - \psi^\ast(y)A(y)\nabla_y \psi(y) \quad \mbox{ for }y\in Y,
\end{equation}
is divergence-free and is of zero average over $Y$, i.e.
\[
\int_Y \beta(y)\, {\rm d}y = 0.
\]
\end{itemize}
The factorization \eqref{eq:factorization} implies that $u_\eps(0,x)$ and $v_\eps(0,x)$ are equal to each other up to a factor of the first eigenfunction $\psi\left(\frac{x}{\eps}\right)$ mentioned above. Observe further that, for the factorization \eqref{eq:factorization} to hold, the factor $v_\varepsilon$ should necessarily satisfy the following initial boundary value problem:
\begin{equation}\label{equation for v-eps}
\left\{
\begin{aligned}
\sigma\left(\frac{x}{\varepsilon}\right)\frac{\partial v_{\varepsilon}}{\partial t} -\operatorname{div}\left(\alpha\left(\frac{x}{\varepsilon}\right)\nabla v_{\varepsilon}\right) + \frac{1}{\varepsilon}\beta\left(\frac{x}{\varepsilon} \right)\cdot\nabla v_{\varepsilon}  & =  0 \qquad \mbox{ in }(0,T)\times \Omega,
\\
v_{\varepsilon}&=0 \qquad \mbox{ on } (0,T]\times \partial \Omega,
\\
v_{\varepsilon}(0,x)&=u_0(x)\qquad \mbox{ in } \Omega,
\end{aligned}
\right.
\end{equation}
where the coefficient vector field $\beta$ is defined in \eqref{eq:defn-beta} and the coefficients $\sigma$ and $\alpha$ are $Y$-periodic and are defined as follows:
\begin{equation}
\sigma(y) := \psi(y)\psi^\ast(y)
\quad
\mbox{ and }
\quad
\alpha(y) := \sigma(y) A(y) 
\quad
\mbox{ for a.e. }y\in Y.
\end{equation}
Note that the coefficient $\sigma$ is $Y$-periodic because
\begin{equation*}
\sigma(y) 
= \psi(y)\psi^\ast(y)
= \psi(y)e^{-2\pi \theta \cdot y} e^{2\pi \theta \cdot y}\psi^\ast(y).
\end{equation*}
\begin{remark}\label{rem:wellprepared}
Our choice of taking the initial datum for $v_\eps$ in \eqref{equation for v-eps} to be $\eps$-independent forces, due to the factorization \eqref{eq:factorization}, the initial datum for $u_\eps$ in \eqref{eq:main} to be of the form:
\begin{equation}\label{eq:wellprepared}
u_\eps(0,x) = u^0_\eps(x) = u_0(x) \psi\left(\frac{x}{\eps}\right)
\end{equation}
This is sometimes referred to as the well prepared initial datum. Similar well preparedness assumptions on the data have been made in the literature \cite{allaire2004homogenization, allaire2007homogenization}. This choice results in strong convergence of the factor $v_\eps$ to its homogenized limit. Note further that \eqref{eq:wellprepared} is the same as 
\[
u_\eps(0,x) = u_0(x) e^{\frac{\theta\cdot x}{\eps}} P\left(\frac{x}{\eps}\right)
\]
for a $Y-$periodic function $P$. Well prepared initial data in this form have appeared in \cite[Remark 3.2]{allaire2007homogenization}, \cite[Theorem 4.4]{MR1796243} and \cite[Theorem 3.2]{allaire2004homogenization}. If one were to keep the data $u_\eps(0,x)$ independent of the parameter $\eps$, then the factorization \eqref{eq:factorization} forces the initial datum for $v_\eps$ to be
\begin{equation}\label{eq:alt-v-eps-init}
v_\eps(0,x) = \frac{u_0(x)}{\psi\left(\frac{x}{\eps}\right)}.
\end{equation}
It should be noted that the equation \eqref{equation for v-eps} can still be homogenized with \eqref{eq:alt-v-eps-init} as the initial datum. The homogenized equation remains the same, but with the initial datum for the limit problem being
\begin{equation*}
\left(\int_Y\frac{1}{\psi(y)}\, {\rm d}y \right) u_0(x).
\end{equation*}
All of this can be argued using the notion of two-scale convergence \cite{allaire1992homogenization}, for example. The authors of \cite{allaire2012homogenization} have studied the case of $\eps$-independent initial data with compact support. They consider a factorisation similar to \eqref{eq:factorization} and find an asymptotic expression for the factor $v_{\eps}$ using the representation formula involving the Green's function. The analysis carried out in \cite{allaire2012homogenization} is quite intricate and makes certain assumptions on the geometry of the support of the initial datum (see \cite[page 305]{allaire2012homogenization}).
\end{remark}
Our main result is the following quantitative homogenization result for the factor $v_\eps$ which solves the initial boundary value problem \eqref{equation for v-eps}:
\begin{theorem}\label{thm:main-result}
Let $v_{\eps}$ be the solution to \eqref{equation for v-eps} and let $v_0$ be the solution to the corresponding homogenized problem \eqref{eq:homogeneqn-v0}. Then there exists a positive constant $C$, independent of $\varepsilon$, such that the following estimate holds:
\begin{equation}\label{main result}
\begin{aligned}
\left \Vert v_\eps - v_0 \right\Vert_{\mathrm L^2(\Omega_T)} 
&  \le C \sqrt{\eps} \left\{ \left\Vert \partial_tv_0\right\Vert_{L^2(0,T;W^{1,d}(\Omega))} + \left\Vert v_0\right\Vert_{\mathrm L^2(0,T;\mathrm H^2(\Omega))} + \left\Vert u_0\right\Vert_{L^2(\Omega)} \right\}.
\end{aligned}
\end{equation}
\end{theorem}
In the above result, there is an inherent assumption that the terms on the right hand side of the bound \eqref{main result} are finite. This essentially translates to the data being sufficient smooth to guarantee adequate regularity for the solution $v_0$ of the homogenized equation. The proof of the above theorem is inspired by the ideas in the work of Jun Geng and Zhongwei Shen \cite{geng2017convergence} that dealt with $\mathrm L^2$ convergence rates in the context of periodic homogenization of parabolic problems. 
\begin{remark}\label{rem:duality-1}
Our estimate \eqref{main result} is not sharp as in \cite{geng2017convergence}. In there, the authors have obtained an $\mathrm L^2$ convergence rate of $\mathcal{O}(\eps)$. Upon deriving a $\mathcal{O}(\sqrt{\eps})$ rate for $\left\Vert v_\eps - v_0\right\Vert_{\mathrm L^2(\Omega_T)}$, the authors in \cite{geng2017convergence} employ a duality argument to improve the rate to $\mathcal{O}(\eps)$. We emphasise that the same strategy cannot be followed in our case as the bound \eqref{main result} involves the $W^{1,d}(\Omega)$-norm. The matter of deriving optimal $\mathcal{O}(\eps)$ rates for $\mathrm L^2$ convergence is left for future investigations. We comment that one may be able to employ the duality argument provided the following bound holds:
\[
\Vert \partial_t v_0 \Vert_{\mathrm L^2(0,T;\mathrm W^{1,d}(\Omega)) } \leq C \Vert f  \Vert_{\mathrm L^2( \Omega_T )}.
\]
\end{remark}
Before we comment further on the proof of Theorem \ref{thm:main-result}, we return to the equation \eqref{equation for v-eps}. As $\beta$ (defined by \eqref{eq:defn-beta}) is periodic, divergence-free and is of zero average, there exists a periodic skew-symmetric matrix-valued function $y\mapsto B(y)$ such that (see \cite[Lemma 1, p. 809]{capdeboscq1998homogenization})
\[
\beta(y) = - \dv_y B(y) \qquad \mbox{ in }Y.
\]
Taking a $Y$-periodic matrix-valued function $y\mapsto M(y) := \alpha(y) + B(y)$, the evolution equation in \eqref{equation for v-eps} can be rewritten as
\begin{equation}\label{equation for v-eps-bis}
\sigma\left(\frac{x}{\varepsilon}\right)\frac{\partial v_\eps}{\partial t} - \dv\left(M\left(\frac{x}{\varepsilon}\right)\nabla v_\eps\right) = 0 \qquad \mbox{ in }(0,T)\times\Omega.
\end{equation}
Observe that $M$ is also elliptic:
\[
M(y)\xi\cdot\xi = \alpha(y)\xi\cdot\xi \ge \mathfrak b \left\vert \xi\right\vert^2 \quad \forall \xi\in\RR^d \, \mbox{ and for a.e. }y\in Y,
\]
where the ellipticity constant $\mathfrak b$ depends on the ellipticity constant $\mu$ of $A$ and the lower bound $a$ in \eqref{eq:eigenfn-normalisation}. Using classical homogenization arguments (see \cite[Theorem 11.4, p. 211]{cioranescu1999introduction}), we can show that 
\begin{equation}\label{eq:veps-to-v0}
\lim_{\eps\to0} \left\Vert v_\eps - v_0 \right\Vert_{\mathrm L^2(\Omega_T)} = 0,
\end{equation}
where $v_0$ satisfies the following homogenized initial boundary value problem:
\begin{equation}\label{eq:homogeneqn-v0}
\left\{
\begin{aligned}
\frac{\partial v_0}{\partial t} - \dv\left(M_h\nabla v_0\right) & =  0 \qquad \mbox{ in }(0,T)\times \Omega,
\\
v_0 & = 0 \qquad \mbox{ on } (0,T]\times \partial \Omega,
\\
v_0(0,x) & = u_0(x) \qquad \mbox{ in } \Omega,
\end{aligned}\right.
\end{equation}
where $M_h$ is the homogenized matrix with constant entries given by
\[
M_h = \int_Y \left( M(y) + M(y) \nabla_y \chi(y) \right)\, {\rm d}y,
\]
with $\chi(y) = \left( \chi_1(y), \dots, \chi_d(y) \right)$, where $\chi_j\in \mathrm H^1_{\text{per}}(Y)$ solve the cell problems:
\begin{equation}\label{eq:cellpb}
\left\{
\begin{aligned}
\dv_y\left(M(y)\nabla_y \left( \chi_j(y) + y_j \right)\right) & = 0 \qquad \mbox{ in }Y,
\\
\int_Y \chi_j(y) \, {\rm d} y & = 0,
\end{aligned}\right.
\end{equation}
for $j=1,\dots,d$. Theorem \ref{thm:main-result} quantifies the aforementioned strong convergence of $v_\eps$ to $v_0$.
\begin{remark}\label{rem:decay-ueps}
The above noted uniform ellipticity of $M$ helps us arrive at the following uniform (with respect to $\eps$) estimate on the solution $v_\eps$ to \eqref{equation for v-eps-bis}:
\begin{equation}\label{eq:LinftyL2-veps}
\left\Vert v_\eps \right\Vert_{\mathrm L^\infty([0,T];\mathrm L^2(\Omega))}
\le C \left\Vert u_0 \right\Vert_{\mathrm L^2(\Omega)}.
\end{equation}
Appealing now to the factorization \eqref{eq:factorization} we arrive at
\begin{equation}\label{eq:ueps-bd}
\left\Vert u_\eps \right\Vert^2_{\mathrm L^2(\Omega_T)} 
= \int_0^T \int_\Omega e^{-\frac{2\lambda t}{\eps^2}} \left\vert \psi\left(\frac{x}{\eps}\right)\right\vert^2 \left\vert v_\eps(t,x) \right\vert^2\, {\rm d}x\, {\rm d}t
\le C \int_0^T e^{-\frac{2\lambda t}{\eps^2}}\, {\rm d}t,
\end{equation}
where we have used the boundedness of the eigenfunction $\psi$ in $\mathrm L^\infty$ and the uniform bound on $v_\eps$ mentioned above. Next, subject to $\lambda$ being strictly positive, performing time integration in the bound \eqref{eq:ueps-bd} leads to 
\begin{equation}\label{eq:decay-ueps}
\left\Vert u_\eps \right\Vert_{\mathrm L^2(\Omega_T)} \le C \eps.
\end{equation}
As mentioned in Remark \ref{rem:wellprepared}, in the absence of the well-prepared initial datum \eqref{eq:wellprepared}, the initial datum for $v_\eps$ in \eqref{equation for v-eps} takes the form \eqref{eq:alt-v-eps-init}. Note that the uniform estimate \eqref{eq:LinftyL2-veps} for $v_\eps$ holds true even in this case. This implies that the above conditional (i.e. under the assumption of $\lambda>0$) decay rate (in $\eps$) for the solution $u_\eps$ to \eqref{eq:main} remains true for general $\mathrm L^2$ initial data (i.e. no need for well-preparedness in the sense of \eqref{eq:wellprepared}). More comments on the aforementioned conditional decay rate for $u_\eps$ is relegated to our concluding remarks (see Section \ref{sec:conclude}). We would like to mention that in \cite{allaire2012homogenization}, the authors have obtained a much better decay rate (albeit a conditional result) of order $\mathcal{O}(\eps^3 \eps^{\frac{d-1}{2}})$ (see \cite[Remark 5.2]{allaire2012homogenization} for further details). 
\end{remark}
\begin{remark}
Factorizations comparable to \eqref{eq:factorization} have appeared before \cite{donato2005averaging, allaireraphael2007homogenization, allaire2007homogenization}. We would like to emphasise here that \cite{donato2005averaging} treats the setting where both the $\eps$-problem and the corresponding homogenized problem are posed in the full space, i.e. the spatial domain is the whole of $\RR^d$. The authors of \cite{allaireraphael2007homogenization} too are working in a perforated medium built out of the whole of $\RR^d$. This is inherently due to the effective non-zero drift present in the model. In \cite{allaire2007homogenization}, the authors pose the problem in a bounded spatial domain $\Omega$. However, they make a change of unknown which takes the domain to the whole space $\RR^d$ and incidentally the homogenized equation obtained in \cite[Theorem 3.1]{allaire2007homogenization} is posed in $\RR^d\times[0,T)$. In \cite{allaire2012homogenization}, however, the case where both the $\eps$-problem and the homogenized problem are posed in $\Omega_T$ is treated. The present work is in similar spirit. Here, we try to quantify the convergence of $v_{\eps}$ to the corresponding homogenized limit as opposed to getting an ansatz for $v_{\eps}$ which is carried out in \cite{allaire2012homogenization}.
\end{remark}
\begin{remark}
Our analysis follows the approach of Geng and Shen \cite{geng2017convergence} which relies on fine properties of a parabolic smoothing operator. If one builds parabolic boundary layer correctors, then the rates can be obtained without the involvement of these smoothing operators. This line of approach is inspired by the work of Allaire and Amar \cite{allaire1999boundary} which treated boundary layers in the context of elliptic problems. We will be addressing these points in a future publication.
\end{remark}
The results of this paper also yields quantitative rates for the periodic homogenization of parabolic equations with nondivergence form elliptic part and with large lower order terms. We draw inspiration from the work of Wenjia Jing and Yiping Zhang \cite{jing2023periodic} which derived quantitative rates for the periodic homogenization of nondivergence form elliptic equations with large drift terms. The key idea in \cite{jing2023periodic} is to use the invariant measure associated with the underlying diffusion process. This is a well-known procedure (see \cite{avellaneda1989compactness, bensoussan2011asymptotic}, also see \cite[section 2.2]{capdeboscq2020finite}) to transform nondivergence form equations into divergence form. This present work deviates from the earlier works by not appealing to this invariant measure. We demonstrate that via exponential eigenvalue problems akin to \eqref{eq:theta-direct-evp}-\eqref{eq:theta-adjoint-evp}, one can transform the parabolic equation with terms in nondivergence form into a parabolic equation with terms in divergence form.   The initial-boundary value problem, for the unknown $\rho_\eps:\overline{\Omega}\times[0,T)\to\RR$, that we examine is the following:
\begin{equation}\label{non-div eq}
\left\{
\begin{aligned}
\frac{\partial \rho_{\eps}}{\partial t} - \mathcal{K}\left(\frac{x}{\eps}\right) : \nabla^2 \rho_\eps + \frac{1}{\eps}q\left(\frac{x}{\eps}\right)\cdot\nabla \rho_\eps + \frac{1}{\eps^2}r\left(\frac{x}{\eps}\right)\rho_{\eps} & = 0 \qquad \mbox{ in }(0,T) \times \Omega,
\\
\rho_{\eps}(t,x) & = 0\qquad \mbox{ on } (0,T)\times \partial \Omega,
\\
\rho_{\eps}(0,x) & = \rho_{\eps}^0(x) \qquad \mbox{ in }\Omega.
 \end{aligned}
\right.
\end{equation}
Here the scalar-valued function $r$, the vector-valued function $q$ and the symmetric matrix-valued function $\mathcal{K}$ are all assumed to be $Y$-periodic. The coefficients $q$ and $r$ are assumed to be bounded in $\mathrm L^\infty(Y)$. We further assume $\mathcal{K}$ to be elliptic and that it belongs to $\mathrm W^{1,s}(Y)$ for some $s>d$. Invoking Sobolev embedding, one gets $\mathcal{K}\in\mathrm C^{0,\tau}(Y)$ for some $\tau\in(0,1]$. Additionally, we shall assume that $\dv_y \mathcal{K}\in \mathrm L^\infty(Y)$. Thanks to these assumptions, the evolution equation in \eqref{non-div eq} can be rewritten as
\begin{equation*}
\frac{\partial \rho_{\eps}}{\partial t} - \dv \left(\mathcal{K}\left(\frac{x}{\eps}\right) \nabla \rho_\eps\right) + \frac{1}{\eps}\left(\dv_y \mathcal{K}\left(\frac{x}{\eps}\right) + q\left(\frac{x}{\eps}\right)\right)\cdot\nabla \rho_\eps + \frac{1}{\eps^2}r\left(\frac{x}{\eps}\right)\rho_{\eps} = 0
\end{equation*}
This is similar to the evolution equation in \eqref{eq:main}. Hence, analogous to the factorization \eqref{eq:factorization}, we shall have that the solution to \eqref{non-div eq} can be factored as follows:
\begin{equation}\label{eq:factorization-nondiv}
\rho_{\varepsilon}(t,x)=e^{\frac{-\widehat\lambda t}{\varepsilon^2}} \varphi_{\varepsilon}(t,x)\widehat\psi\left(\frac{x}{\varepsilon}\right),
\end{equation}
where $\widehat\lambda$ and $\widehat\psi$ are an eigenpair associated with the exponential eigenvalue problems akin to \eqref{eq:theta-direct-evp}-\eqref{eq:theta-adjoint-evp}, with the following choices of the coefficients:
\begin{equation*}
A(y) = A^\ast(y) = \mathcal{K}(y);
\qquad
b(y) = \dv_y\mathcal{K}(y) + q(y);
\qquad
c(y) = r(y).
\end{equation*}
Thus, Theorem \ref{thm:main-result} gives a quantitative rate of convergence for $\varphi_\eps$ to the solution of the corresponding homogenized equation.
\begin{remark}
In the context of the quantitative periodic homogenization of nondivergence form elliptic equations, Jing and Zhang \cite{jing2023periodic} left open the case involving potential terms of order $\mathcal{O}(\eps^{-2})$. We point out that the present work does not give a definitive answer to that question. Nevertheless, we have treated the parabolic version of that case. The technique adapted in this work is tailored specifically for parabolic problems. We cite the work of Zhang \cite{zhang2021estimates} which treated the large potential case in divergence form (without the drift terms), but of $\mathcal{O}(\eps^{-1})$.
\end{remark}
We finish this rather long introduction by detailing the structure of the rest of the paper. Section \ref{sec:2} quantifies the convergence \eqref{eq:veps-to-v0} for general second order parabolic problems with the positive coefficient next to the time derivative term being merely measurable and bounded. The main result of section \ref{sec:2} is Theorem \ref{thm:homogen-feps} which gives a rate of $\mathcal{O}(\eps^{\frac14})$. Section \ref{sec:3} improves this rate to $\mathcal{O}(\eps^\frac12)$ with an additional assumption that the coefficient of the time derivative is regular (see Theorem \ref{eq:L2-estimate better}). Section \ref{sec:conclude} presents few scenarios where decay rate mentioned in Remark \ref{rem:decay-ueps} holds true. The manuscript concludes with an appendix which gathers some useful results on the parabolic smoothing operator.
\section*{Acknowledgements}
The author would like to thank his PhD advisor Harsha Hutridurga for introducing him to the exciting theory of homogenization and for his useful remarks during the preparation of this manuscript. The author has greatly benefited from the discussions with Chandan Biswas and Debanjana Mitra on a preliminary version of this manuscript. The author would also like to thank Gr\'egoire Allaire for pointing him to the reference \cite{allaire2012homogenization}. Finally,  the author acknowledges the support of NBHM-DAE (Government of India).
\section{The case of merely bounded coefficients}\label{sec:2}
In this and the next section, we will study the following general parabolic problem:
\begin{equation}\label{eq:equation-for-feps}
\left\{
\begin{aligned}
\zeta\left(\frac{x}{\eps}\right)\frac{\partial f_{\varepsilon}}{\partial t} - \dv\left(\Theta\left(\frac{x}{\eps}\right)\nabla f_{\varepsilon}\right) & =  0 \qquad \mbox{ in }(0,T)\times \Omega,
\\
f_\eps & = 0 \qquad \mbox{ on } (0,T]\times \partial \Omega,
\\
f_\eps(0,x) & = f_{\text{in}}(x) \qquad \mbox{ in } \Omega.
\end{aligned}
\right.
\end{equation}
We assume that the coefficient $\zeta\in\mathrm L^{\infty}_{\text{per}}(Y)$ satisfies
\begin{equation}\label{eq:zeta-average-1}
\zeta(y) \ge c >0 \, \mbox{ for a.e. }y\in Y
\quad\mbox{ and }\quad
\int_Y \zeta(y)\, {\rm d}y = 1,
\end{equation}
where $c$ is a constant. In the next section, an additional regularity assumption is made on $\zeta$. The matrix-valued coefficient $\Theta\in\mathrm L^\infty_{\text{per}}(Y)$ is assumed to be uniformly elliptic, i.e.
\[
\Theta(y)\xi\cdot\xi \ge \kappa \left\vert \xi\right\vert^2 \quad \forall \xi\in\RR^d \, \mbox{ and for a.e. }y\in Y,
\]
for some $\kappa>0$. Considering the solutions $\omega_j\in \mathrm H^1_{\text{per}}(Y)$ to the cell problems
\begin{equation}\label{eq:cellpb-Theta}
\left\{
\begin{aligned}
\dv_y\left(\Theta(y)\nabla_y \left( \omega_j(y) + y_j \right)\right) & = 0 \qquad \mbox{ in }Y,
\\
\int_Y \zeta(y)\omega_j(y) \, {\rm d} y & = 0,
\end{aligned}\right.
\end{equation}
for $j=1,\dots,d$, we can define a homogenized matrix as follows:
\[
\Theta_h = \int_Y \left( \Theta(y) + \Theta(y) \nabla_y \omega(y) \right)\, {\rm d}y,
\]
with $\omega(y) = \left( \omega_1(y), \dots, \omega_d(y) \right)$. Next, we consider the homogenized problem:
\begin{equation}\label{eq:homogeneqn-f0}
\left\{
\begin{aligned}
\frac{\partial f_0}{\partial t} - \dv\left(\Theta_h\nabla f_0\right) & =  0 \qquad \mbox{ in }(0,T)\times \Omega,
\\
f_0 & = 0 \qquad \mbox{ on } (0,T]\times \partial \Omega,
\\
f_0(0,x) & = f_{\text{in}}(x) \qquad \mbox{ in } \Omega.
\end{aligned}\right.
\end{equation}
The following is a quantitative result corresponding to the periodic homogenization of \eqref{eq:equation-for-feps}.
\begin{theorem}\label{thm:homogen-feps}
Suppose $\zeta \in \mathrm L_{per}^{\infty}(Y)$ satisfies \eqref{eq:zeta-average-1}.
Let $f_{\eps}$ be the solution to \eqref{eq:equation-for-feps} and let $f_0$ be the solution to the corresponding homogenized problem \eqref{eq:homogeneqn-f0}. Then there exists a positive constant $C$, independent of $\varepsilon$, such that the following estimate holds:
\begin{equation}\label{eq:L2-estimate}
\begin{aligned}
\left\Vert f_\eps - f_0 \right\Vert_{\mathrm L^2(\Omega_T)} 
&  \le C \eps^{\frac14} \left\{ \left\Vert \partial_tf_0\right\Vert_{L^2(0,T;W^{1,d}(\Omega))} + \left\Vert f_0\right\Vert_{\mathrm L^2(0,T;\mathrm H^2(\Omega))} + \left\Vert f_{{\rm{in}}}\right\Vert_{L^2(\Omega)} \right\}.
\end{aligned}
\end{equation}
\end{theorem}
Our line of approach is inspired by the work of Geng and Shen \cite{geng2017convergence}. We introduce the following notation: for $s>0$,
\[
\Omega_s := \left\{ x\in \Omega \, \mbox{ such that }\mbox{ dist}(x,\partial\Omega)<s\right\}.
\]
We fix a space cut-off function $\eta_1^\eps \in \mathrm{C}_c^{\infty}(\Omega)$ such that $0\leq \eta_1^\eps \leq 1$ and it further satisfies:
\[
\eta_1^\eps = 1 \quad \mbox{ in }\quad \Omega \backslash \Omega_{4\sqrt{\varepsilon}},
\qquad
\eta_1^\eps = 0 \quad \mbox{ in }\quad \Omega_{3\sqrt{\varepsilon}},
\qquad
\left\vert\nabla \eta_1^\eps \right\vert \leq \frac{C}{\sqrt{\eps}} \quad \mbox{ in }\quad \Omega.
\]
We also fix a time cut-off function $\eta_2^\eps \in \mathrm{C}_c^\infty(0,T)$ such that $0\leq \eta_2^\eps \leq 1$ and it further satisfies:
\[
\eta_2^\eps = 1 \quad \mbox{ in }\, \left(8\eps, T-8\eps\right),
\qquad
\eta_2^\eps = 0 \quad \mbox{ in }\, (0,4\eps]\cup(T-4\eps,T),
\qquad 
\left\vert \partial_t \eta_2^\eps\right\vert \leq \frac{C}{\eps} \quad \mbox{ in }\, (0,T).
\]
Next, we consider a parabolic smoothing operator borrowed from \cite[Section 3, p.2102]{geng2017convergence}:
\begin{equation}\label{eq:smoothing-op}
S_\eps(g)(t,x) := \frac{1}{\eps^{d+2}} \int_{\RR^{d+1}} g(t-s,x-y)\, \theta\left(\frac{s}{\eps^2}, \frac{y}{\eps}\right)\, {\rm d}y\, {\rm d}s,
\end{equation}
where $\theta\in \mathrm C^\infty_c(\mathcal{N})$ is a fixed nonnegative function that satisfies
\[
\int_{\RR^{d+1}}\theta(s,y) \, {\rm d}y\, {\rm d}s = 1.
\]
Here the set $\mathcal{N}\subset\RR^{d+1}$ is defined below:
\[
\mathcal{N} := \left\{ (t,x)\in\RR\times\RR^{d} \mbox{ such that }\left\vert t\right\vert + \left\vert x\right\vert^2 \le 1 \right\}.
\]
We shall also define the following subset of $\Omega\times(0,T)$:
\[
\Omega_{T,\delta} := \Big( \Omega_\delta \times (0,T) \Big) \cup \Big( \Omega \times (0,\delta^2) \Big) \cup \Big( \Omega \times (T-\delta^2, T) \Big) \qquad \mbox{ for }\delta>0.
\]
In the appendix, we have gathered few useful properties of the smoothing operator $S_\eps$ which play a crucial role in our proof. Using the cut-off functions and the smoothing operator mentioned above, we define $w_\eps:\Omega_T\to\RR$ as follows:
\begin{equation}\label{eq:defn-weps}
w_\eps := f_\eps - f_0 - \eps \omega\left(\frac{x}{\eps}\right)\cdot S_\eps \left( \eta_1^\eps \eta_2^\eps \nabla f_0 \right),
\end{equation}
where $f_\eps$ and $f_0$ are solutions to \eqref{eq:equation-for-feps} and \eqref{eq:homogeneqn-f0} respectively. Note that $w_\eps\in\mathrm L^2(0,T;\mathrm H^1_0(\Omega))$ and that $w_\eps(0,\cdot)\equiv0$. The study of a function defined similar to \eqref{eq:defn-weps} is carried out in the quantitative analysis of periodic elliptic problems \cite[Chapter 3, page 33]{shen2018periodic}. We note, however, that the analysis carried out by Geng and Shen in \cite[equation (1.14), page 2095]{geng2017convergence} considers a function which also involves a term of order $\mathcal{O}(\eps^2)$. We will be employing such a construction in the next section. In the remainder of this manuscript, we will sometimes be using the following shorthand notations:
\[
\zeta_\eps(x) := \zeta\left(\frac{x}{\eps}\right),
\quad
\Theta_\eps(x) := \Theta\left(\frac{x}{\eps}\right),
\quad
\omega_\eps(x) := \omega\left(\frac{x}{\eps}\right)
\]
We begin our analysis with the following result. 
\begin{lemma}
Let $w_\eps$ be defined by \eqref{eq:defn-weps}. Then for any $\psi\in\mathrm L^2(0,T;\mathrm H^1_0(\Omega))$, we have
\begin{equation}\label{eq:wkform-lemma}
\begin{aligned}
\int_0^T & \left\langle \zeta_\eps\partial_{t}w_{\varepsilon},\psi \right\rangle_{\mathrm{H}^{-1}(\Omega) \times \mathrm{H}_0^1(\Omega)} 
+ \int_{\Omega_T} \Theta_\eps\nabla w_{\varepsilon} \cdot \nabla \psi 
= - \int_{\Omega_T} (\zeta_\eps -1)\partial_{t}f_{0}\psi 
\\
& \quad - \eps \int_{\Omega_T} \zeta_\eps \omega_\eps \cdot \partial_t \left[S_\eps (\eta_1^\eps \eta_2^\eps \nabla f_0)\right] \psi 
+\int_{\Omega_T}\left(\Theta_h - \Theta_\eps\right) \left\{\nabla f_0 - S_\eps(\eta_1^\eps\eta_2^\eps \nabla f_0)\right\} \cdot \nabla \psi
\\
& \quad 
- \eps \int_{\Omega_T} \Theta_\eps \nabla \bigg\{S_{\varepsilon}(\eta_1^\eps \eta_2^\eps \nabla f_0)\bigg\}\omega_\eps \cdot \nabla \psi
+ \eps \int_{\Omega_T} \phi_{kij}\left(\frac{x}{\eps}\right) \frac{\partial}{\partial x_k} \bigg\{S_{\varepsilon}\left(\eta_1^\eps \eta_2^\eps \frac{\partial f_0}{\partial x_j} \right)\bigg\} \frac{\partial\psi}{\partial x_i},
\end{aligned}
\end{equation}
where the summation convention over repeated indices $i,j,k\in\{1,\dots,d\}$ is used.\\
Here $\phi_\eps(x) = \phi \left(\frac{x}{\eps}\right)$ with $\phi$ being a $Y$-periodic third-order anti-symmetric tensor such that
\begin{equation}\label{eq:div-phi-B}
\dv_y \phi(y) = \Theta(y) + \Theta(y)\nabla_y\omega(y) - \Theta_h \qquad \mbox{ in }Y.
\end{equation}
\end{lemma}
\begin{proof}
A direct computation leads to
\[
\zeta_\eps\partial_{t}w_{\varepsilon}
= \zeta_\eps \partial_t f_\eps 
- \zeta_\eps \partial_t f_0 
- \eps \zeta_\eps \omega_\eps\cdot \partial_t \left\{S_\eps \left( \eta_1^\eps \eta_2^\eps \nabla f_0 \right)\right\}
\]
and
\begin{equation*}
\begin{aligned}
\Theta_\eps \nabla w_\eps 
& = \Theta_\eps \nabla f_\eps
- \Theta_\eps \nabla f_0
- \Theta_\eps \nabla_y \omega\left(\frac{x}{\eps}\right) S_\eps \left( \eta_1^\eps \eta_2^\eps \nabla f_0 \right)
- \eps \Theta_\eps \nabla \left\{S_\eps \left( \eta_1^\eps \eta_2^\eps \nabla f_0 \right)\right\} \omega_\eps
\\
& = \Theta_\eps \nabla f_\eps
- \Theta_h \nabla f_0 
+ \left(\Theta_h - \Theta_\eps\right) \nabla f_0
- \left( \Theta_\eps + \Theta_\eps \nabla_y \omega\left(\frac{x}{\eps}\right) - \Theta_h\right) S_\eps \left( \eta_1^\eps \eta_2^\eps \nabla f_0 \right)
\\
& \quad - \left(\Theta_h - \Theta_\eps\right) S_\eps \left( \eta_1^\eps \eta_2^\eps \nabla f_0 \right)
- \eps \Theta_\eps \nabla \left\{S_\eps \left( \eta_1^\eps \eta_2^\eps \nabla f_0 \right)\right\} \omega_\eps
\\
& = \Theta_\eps \nabla f_\eps
- \Theta_h \nabla f_0 
+ \left(\Theta_h - \Theta_\eps\right) \left( \nabla f_0 - S_\eps \left( \eta_1^\eps \eta_2^\eps \nabla f_0 \right) \right)
- \eps \dv \left( \phi_\eps \right) S_\eps \left( \eta_1^\eps \eta_2^\eps \nabla f_0 \right) 
\\
& \quad 
- \eps \Theta_\eps \nabla \left\{S_\eps \left( \eta_1^\eps \eta_2^\eps \nabla f_0 \right)\right\} \omega_\eps,
\end{aligned}
\end{equation*}
where we have used the equation \eqref{eq:div-phi-B} in the last step. Hence, for every $\psi\in\mathrm L^2(0,T;\mathrm H^1_0(\Omega))$, we have
\begin{equation}\label{eq:wkform-1}
\begin{aligned}
\int_0^T & \left\langle \zeta_\eps\partial_{t}w_{\varepsilon},\psi \right\rangle_{\mathrm{H}^{-1}(\Omega) \times \mathrm{H}_0^1(\Omega)} 
+ \int_{\Omega_T} \Theta_\eps\nabla w_{\varepsilon} \cdot \nabla \psi 
= - \int_{\Omega_T} (\zeta_\eps -1)\partial_{t}f_{0}\psi  
\\
& \qquad - \eps \int_{\Omega_T} \zeta_\eps \omega_\eps \cdot \partial_t \left[S_\eps (\eta_1^\eps \eta_2^\eps \nabla f_0)\right] \psi 
+\int_{\Omega_T}\left(\Theta_h - \Theta_\eps\right) \left\{\nabla f_0 - S_\eps(\eta_1^\eps\eta_2^\eps \nabla f_0)\right\} \cdot \nabla \psi
\\
& \qquad - \eps \int_{\Omega_T} \Theta_\eps \nabla \bigg\{S_{\varepsilon}(\eta_1^\eps \eta_2^\eps \nabla f_0)\bigg\}\omega_\eps \cdot \nabla \psi
- \eps \int_{\Omega_T} \dv \left(\phi_\eps\right) \bigg\{S_{\varepsilon}(\eta_1^\eps \eta_2^\eps \nabla f_0)\bigg\} \cdot \nabla \psi,
\end{aligned}
\end{equation}
where we have used the fact that $f_\eps$ and $f_0$ satisfy the evolution equations in \eqref{eq:equation-for-feps} and \eqref{eq:homogeneqn-f0}, respectively. Owing to the anti-symmetric nature of the third-order tensor (i.e. $\phi_{kij} = - \phi_{ikj}$ for $1\le i,j,k \le d$), it follows that (see \cite[Remark 3.1.3]{shen2018periodic}) for a nice function $\psi$, we have
\[
\frac{\partial}{\partial x_k} \left( \phi_{kij} \right) \frac{\partial \psi}{\partial x_i} = \frac{\partial}{\partial x_k} \left( \phi_{kij} \frac{\partial \psi}{\partial x_i} \right).
\]
Thus, an integration by parts in the last term of the right hand side of \eqref{eq:wkform-1} results in \eqref{eq:wkform-lemma}.
\end{proof}
\begin{remark}\label{rem:zero-avg-periodic}
Observe that for any $\tau\in\mathrm L^2_{\textrm{per}}(Y)$ satisfying 
\[
\int_Y\tau(y)\, {\rm d}y = 0,
\]
there exists a $\varpi\in \mathrm H^2_{\text{per}}(Y)$ such that
\[
\Delta_y \varpi = \tau \quad \mbox{ in }Y.
\]
Taking $g=\nabla_y \varpi$, we deduce that for any $\kappa\in \mathrm H^1_0(\Omega)$,
\[
\int_\Omega \tau\left(\frac{x}{\eps}\right) \kappa(x)\, {\rm d}x
= \eps \int_\Omega \dv\left(g\left(\frac{x}{\eps}\right)\right) \kappa(x) \, {\rm d}x
= - \eps \int_\Omega g\left(\frac{x}{\eps}\right) \cdot \nabla \kappa(x)\, {\rm d}x.
\]
As Sobolev embedding implies that $g\in \mathrm L^{\frac{2d}{d-2}}(\Omega)$, we can get the following bound:
\begin{equation}\label{eq:bound-rmk}
\left\vert \int_\Omega \tau\left(\frac{x}{\eps}\right) \kappa(x)\, {\rm d}x \right\vert 
\le C \eps \left\Vert \nabla \kappa\right\Vert_{\mathrm L^{\frac{2d}{d+2}}(\Omega)}.
\end{equation}
\end{remark}
In the next lemma, we obtain a bound for the left hand side of \eqref{eq:wkform-lemma}.
\begin{lemma}\label{lem:wkform-bound}
Let $w_\eps$ be defined by \eqref{eq:defn-weps}. Then for any $\psi\in\mathrm L^2(0,T;\mathrm H^1_0(\Omega))$, we have
\begin{equation}\label{eq:wkform-bound}
\begin{aligned}
\int_0^T & \left\langle \zeta_\eps\partial_{t}w_{\varepsilon},\psi \right\rangle_{\mathrm{H}^{-1}(\Omega) \times \mathrm{H}_0^1(\Omega)} 
+ \int_{\Omega_T} \Theta_\eps\nabla w_{\varepsilon} \cdot \nabla \psi 
\\
& \le C \eps^{\frac14} \left\{ \left\Vert \partial_tf_0\right\Vert_{L^2(0,T;W^{1,d}(\Omega))} + \left\Vert f_0\right\Vert_{\mathrm L^2(0,T;\mathrm H^2(\Omega))} + \left\Vert f_{{\rm{in}}}\right\Vert_{L^2(\Omega)} \right\}
\left\Vert \psi\right\Vert_{\mathrm L^2(0,T;\mathrm H^1_0(\Omega))}
\end{aligned}
\end{equation}
 \end{lemma}
\begin{proof}
It follows from \eqref{eq:wkform-lemma} that
\begin{equation*}
\begin{aligned}
\int_0^T & \left\langle \zeta_\eps\partial_{t}w_{\varepsilon},\psi \right\rangle_{\mathrm{H}^{-1}(\Omega) \times \mathrm{H}_0^1(\Omega)} 
+ \int_{\Omega_T} \Theta_\eps\nabla w_{\varepsilon} \cdot \nabla \psi 
\le \left\vert \int_{\Omega_T} (\zeta_\eps -1)\partial_{t}f_{0}\psi \right\vert
\\
& \quad + \eps \left\vert \int_{\Omega_T} \zeta_\eps \omega_\eps \cdot \partial_t \left[S_\eps (\eta_1^\eps \eta_2^\eps \nabla f_0)\right] \psi \right\vert
+\left\vert \int_{\Omega_T}\left(\Theta_h - \Theta_\eps\right) \left\{\nabla f_0 - S_\eps(\eta_1^\eps\eta_2^\eps \nabla f_0)\right\} \cdot \nabla \psi\right\vert
\\
& \quad 
+ \eps \left\vert \int_{\Omega_T} \Theta_\eps \nabla \bigg\{S_{\varepsilon}(\eta_1^\eps \eta_2^\eps \nabla f_0)\bigg\}\omega_\eps \cdot \nabla \psi\right\vert
+ \eps \left\vert \int_{\Omega_T} \phi_{kij}\left(\frac{x}{\eps}\right) \frac{\partial}{\partial x_k} \bigg\{S_{\varepsilon}\left(\eta_1^\eps \eta_2^\eps \frac{\partial f_0}{\partial x_j} \right)\bigg\} \frac{\partial\psi}{\partial x_i} \right\vert
\\
& \quad =: I_1 + I_2 + I_3 + I_4 + I_5
\end{aligned}
\end{equation*}
Thanks to the assumption \eqref{eq:zeta-average-1} on $\zeta$, invoking \eqref{eq:bound-rmk} from Remark \ref{rem:zero-avg-periodic} we arrive at the following bound:
\begin{equation*}
\begin{aligned}
I_1 
\le C \eps \int_0^T & \left\Vert \nabla \left( \partial_t f_0 \psi \right) \right\Vert_{\mathrm L^{\frac{2d}{d+2}}(\Omega)}
\le C \eps \left\{ \int_0^T \left\Vert \psi \nabla\partial_t f_0 \right\Vert_{\mathrm L^{\frac{2d}{d+2}}(\Omega)} 
+ \int_0^T \left\Vert \partial_t f_0 \nabla \psi \right\Vert_{\mathrm L^{\frac{2d}{d+2}}(\Omega)}\right\}
\\
& \le C \eps \left\{ \int_0^T \left\Vert \nabla\partial_t f_0 \right\Vert_{\mathrm L^d(\Omega)} \left\Vert \psi \right\Vert_{\mathrm L^2(\Omega)}  
+ \int_0^T \left\Vert \partial_t f_0 \right\Vert_{\mathrm L^d(\Omega)} \left\Vert \nabla \psi \right\Vert_{\mathrm L^2(\Omega)}\right\}
\\
& \le C \eps \left\{ \int_0^T \left\Vert \nabla\partial_t f_0 \right\Vert_{\mathrm L^d(\Omega)} \left\Vert \nabla \psi \right\Vert_{\mathrm L^2(\Omega)}  
+ \int_0^T \left\Vert \partial_t f_0 \right\Vert_{\mathrm L^d(\Omega)} \left\Vert \nabla \psi \right\Vert_{\mathrm L^2(\Omega)}\right\}
\\
& \le C \eps \left\{ \left( \int_0^T \left\Vert \nabla\partial_t f_0 \right\Vert^2_{\mathrm L^d(\Omega)}\right)^\frac12
+ \left( \int_0^T \left\Vert \partial_t f_0 \right\Vert^2_{\mathrm L^d(\Omega)} \right)^\frac12
\right\}
\left\Vert \psi\right\Vert_{\mathrm L^2(0,T;\mathrm H^1_0(\Omega))}
\\
& \le C \eps\left\Vert \partial_t f_0 \right\Vert_{L^2(0,T;\mathrm W^{1,d}(\Omega))} \left\Vert \psi\right\Vert_{\mathrm L^2(0,T;\mathrm H^1_0(\Omega))} .
\end{aligned}
\end{equation*}
Note that we have employed the H\"older inequality in the third and the fifth step while the Poincar\'e inequality in $\mathrm H^1_0(\Omega)$ is used in the fourth step. In order to handle the $I_2$ term, we recall from \eqref{eq:cellpb-Theta} that the product $\zeta \omega$ averages to zero over the unit cell $Y$. Thus, Remark \ref{rem:zero-avg-periodic} yields the existence of a matrix-valued function $\mathcal{D}$ (need not be skew-symmetric) bounded in $\mathrm L^\infty(Y)$ (thanks to elliptic regularity) such that
\begin{equation*}
\eps \int_{\Omega_T} \zeta_\eps \omega_\eps \cdot \partial_t \left\{S_\eps (\eta_1^\eps \eta_2^\eps \nabla f_0)\right\} \psi 
= \eps^2 \int_{\Omega_T} \frac{\partial}{\partial x_i} \left( \mathcal{D}_{ij} \left(\frac{x}{\eps}\right)\right) \partial_t \left[S_\eps \left(\eta_1^\eps \eta_2^\eps \frac{\partial f_0}{\partial x_j}\right)\right] \psi
\end{equation*}
On the right hand side of the above expression, performing an integration by parts leads to
\begin{equation*}
- \eps^2 \int_{\Omega_T} \mathcal{D}_{ij} \left(\frac{x}{\eps}\right) \frac{\partial^2}{\partial x_i\partial t}\left[S_\eps \left(\eta_1^\eps \eta_2^\eps \frac{\partial f_0}{\partial x_j}\right)\right] \psi
- \eps^2 \int_{\Omega_T} \mathcal{D}_{ij} \left(\frac{x}{\eps}\right) \partial_t\left[S_\eps \left(\eta_1^\eps \eta_2^\eps \frac{\partial f_0}{\partial x_j}\right)\right] \frac{\partial \psi}{\partial x_i}
\end{equation*}
Hence it follows that
\begin{equation*}
\begin{aligned}
I_2 & \le \eps^2 \left\vert \int_{\Omega_T} \mathcal{D}_{ij} \left(\frac{x}{\eps}\right) \frac{\partial^2}{\partial x_i\partial t}\left[S_\eps \left(\eta_1^\eps \eta_2^\eps \frac{\partial f_0}{\partial x_j}\right)\right] \psi \right\vert
+ \eps^2 \left\vert \int_{\Omega_T} \mathcal{D}_{ij} \left(\frac{x}{\eps}\right) \partial_t\left[S_\eps \left(\eta_1^\eps \eta_2^\eps \frac{\partial f_0}{\partial x_j}\right)\right] \frac{\partial \psi}{\partial x_i} \right\vert
\\
& =: J_1 + J_2
\end{aligned}
\end{equation*}
Note that 
\begin{equation*}
\begin{aligned}
\partial_t\left[S_\eps \left(\eta_1^\eps \eta_2^\eps \frac{\partial f_0}{\partial x_j}\right)\right] 
& = S_\eps \left(\partial_t\left(\eta_1^\eps \eta_2^\eps\right) \frac{\partial f_0}{\partial x_j}\right)
+ S_\eps \left(\eta_1^\eps \eta_2^\eps \frac{\partial^2 f_0}{\partial t\partial x_j}\right)
\\
& = S_\eps \left(\partial_t\left(\eta_1^\eps \eta_2^\eps\right) \frac{\partial f_0}{\partial x_j}\right)
+ \frac{\partial}{\partial x_j} S_\eps \left(\eta_1^\eps \eta_2^\eps \frac{\partial f_0}{\partial t}\right)
- S_\eps \left(\frac{\partial}{\partial x_j}\left(\eta_1^\eps \eta_2^\eps\right) \frac{\partial f_0}{\partial t}\right).
\end{aligned}
\end{equation*}
Therefore, we have
\begin{equation*}
\begin{aligned}
J_2 & \le C \eps^2 \Bigg\{ 
\left\Vert S_\eps \Big(\partial_t\left(\eta_1^\eps \eta_2^\eps\right) \nabla f_0\Big) \right\Vert_{\mathrm L^2(\Omega_T)}
+ \left\Vert \nabla \left( S_\eps \left(\eta_1^\eps \eta_2^\eps \frac{\partial f_0}{\partial t}\right) \right)\right\Vert_{\mathrm L^2(\Omega_T)}
\\
& \hspace{6.0 cm}+ \left\Vert S_\eps \left(\nabla\left(\eta_1^\eps \eta_2^\eps\right) \frac{\partial f_0}{\partial t}\right)\right\Vert_{\mathrm L^2(\Omega_T)}
\Bigg\} \left\Vert \nabla\psi \right\Vert_{\mathrm L^2(\Omega_T)}
\\
& \le C \eps^2 \Bigg\{ \frac{1}{\eps} \left\Vert \nabla f_0 \right\Vert_{\mathrm L^2(\Omega_T)}
+ \frac{1}{\eps} \left\Vert \frac{\partial f_0}{\partial t} \right\Vert_{\mathrm L^2(\Omega_T)}
+ \frac{1}{\sqrt{\eps}} \left\Vert \frac{\partial f_0}{\partial t} \right\Vert_{\mathrm L^2(\Omega_T)}
\Bigg\} \left\Vert \nabla\psi \right\Vert_{\mathrm L^2(\Omega_T)},
\end{aligned}
\end{equation*}
where we have used the Lemma \ref{smooth: compact support} from the Appendix. Note that
\begin{equation*}
\begin{aligned}
\frac{\partial^2}{\partial x_i\partial t} & \left[S_\eps \left(\eta_1^\eps \eta_2^\eps \frac{\partial f_0}{\partial x_j}\right)\right] 
= S_\eps \left(\frac{\partial^2}{\partial x_i\partial t}\left(\eta_1^\eps \eta_2^\eps\right) \frac{\partial f_0}{\partial x_j}\right)
+ S_\eps \left(\frac{\partial}{\partial t}\left(\eta_1^\eps \eta_2^\eps\right) \frac{\partial^2 f_0}{\partial x_i \partial x_j}\right)
\\
& \qquad \qquad \qquad \qquad + S_\eps \left(\frac{\partial}{\partial x_i}\left(\eta_1^\eps \eta_2^\eps\right) \frac{\partial^2 f_0}{\partial t \partial x_j}\right) 
+ S_\eps \left( \eta_1^\eps \eta_2^\eps \frac{\partial}{\partial x_i} \left(\frac{\partial^2 f_0}{\partial t \partial x_j}\right)\right)
\\
& = S_\eps \left(\frac{\partial^2}{\partial x_i\partial t}\left(\eta_1^\eps \eta_2^\eps\right) \frac{\partial f_0}{\partial x_j}\right)
+ \frac{\partial}{\partial x_i} \left( S_\eps \left( \eta_1^\eps \eta_2^\eps \frac{\partial^2 f_0}{\partial t \partial x_j}\right)\right)
+ S_\eps \left(\frac{\partial}{\partial t}\left(\eta_1^\eps \eta_2^\eps\right) \frac{\partial^2 f_0}{\partial x_i \partial x_j}\right).
\end{aligned}
\end{equation*}
Hence we have
\begin{equation*}
\begin{aligned}
J_1 & \le C \eps^2 \Bigg\{
\left\Vert S_\eps \Big(\partial_t\nabla \left(\eta_1^\eps \eta_2^\eps\right) \nabla f_0 \Big)\right\Vert_{\mathrm L^2(\Omega_T)}
+ \left\Vert \nabla \left(S_\eps \Big(\eta_1^\eps \eta_2^\eps \partial_t\nabla f_0 \Big)\right)\right\Vert_{\mathrm L^2(\Omega_T)}
\\
& \hspace{6.5 cm}+ \left\Vert S_\eps \Big( \partial_t \left(\eta_1^\eps \eta_2^\eps\right) \nabla^2 f_0 \Big)\right\Vert_{\mathrm L^2(\Omega_T)}
\Bigg\} \left\Vert \psi \right\Vert_{\mathrm L^2(\Omega_T)}
\\
& \le C \eps^2 \Bigg\{ \frac{1}{\eps^{\frac32}} \left\Vert \nabla f_0 \right\Vert_{\mathrm L^2(\Omega_T)}
+ \frac{1}{\eps} \left\Vert \partial_t\nabla f_0 \right\Vert_{\mathrm L^2(\Omega_T)}
+ \frac{1}{\eps} \left\Vert \nabla^2 f_0 \right\Vert_{\mathrm L^2(\Omega_T)}
\Bigg\} \left\Vert \psi \right\Vert_{\mathrm L^2(\Omega_T)}.
\end{aligned}
\end{equation*}
The above obtained estimates on $J_1$ and $J_2$ together will result in the following bound for $I_2$:
\begin{equation*}
I_2 \leq C \sqrt{\eps} \big\{ \left\Vert \partial_t f_0 \right\Vert_{\mathrm L^2(0,T;\mathrm W^{1,d}(\Omega))}  + \left\Vert f_0 \right\Vert_{\mathrm L^2(0,T;\mathrm H^2(\Omega))} \big\} \left\Vert \nabla \psi \right\Vert_{\mathrm L^2(\Omega_T)}
\end{equation*}
Note that the product $\eta_1^\eps(x)\eta_2^\eps(t)$ takes the value one on $\Omega_T\setminus \Omega_{T,4\sqrt{\eps}}$. Hence we have
\begin{equation*}
\begin{aligned}
I_3 & \le \left\vert \int_{\Omega_{T,4\sqrt{\eps}}}\left(\Theta_h - \Theta_\eps\right) \left\{\nabla f_0 - S_\eps(\eta_1^\eps\eta_2^\eps \nabla f_0)\right\} \cdot \nabla \psi\right\vert
+ \left\vert \int_{\Omega_T\setminus\Omega_{T,4\sqrt{\eps}}}\left(\Theta_h - \Theta_\eps\right) \left\{\nabla f_0 - S_\eps(\nabla f_0)\right\} \cdot \nabla \psi\right\vert
\\
& \le C  \left\Vert \nabla f_0 - S_\eps(\eta_1^\eps\eta_2^\eps \nabla f_0) \right\Vert_{\mathrm L^2(\Omega_{T,4\sqrt{\eps}})}
\left\Vert \nabla \psi \right\Vert_{\mathrm L^2(\Omega_{T,4\sqrt{\eps}})}
\\
& \hspace{6.0cm} + C \left\Vert \nabla f_0 - S_\eps(\nabla f_0) \right\Vert_{\mathrm L^2(\Omega_T\setminus\Omega_{T,4\sqrt{\eps}})}
\left\Vert \nabla \psi \right\Vert_{\mathrm L^2(\Omega_T\setminus\Omega_{T,4\sqrt{\eps}})}
\\
& =: J_3 + J_4
\end{aligned}
\end{equation*}
Invoking the Calder\'on's extension theorem for $f_0\in \mathrm L^2(0,T;\mathrm H^2(\Omega))\cap\mathrm H^1(0,T;\mathrm L^2(\Omega))$, we can find a $\widetilde{f_0}:\RR^{d+1}\to\RR$ such that
\begin{equation*}
\left\Vert \nabla^2 \widetilde{f_0} \right\Vert_{\mathrm L^2(\RR^{d+1})}
+ \left\Vert \partial_t \widetilde{f_0} \right\Vert_{\mathrm L^2(\RR^{d+1})}
\le C \Big\{
\left\Vert f_0 \right\Vert_{\mathrm L^2(0,T;\mathrm H^2(\Omega))} 
+ \left\Vert \partial_t f_0 \right\Vert_{\mathrm L^2(\Omega_T)}
\Big\}
\end{equation*}
Using this, we can bound $J_4$ as follows:
\begin{equation*}
\begin{aligned}
J_4 & \le C \left\Vert \nabla \widetilde{f_0} - S_\eps(\nabla \widetilde{f_0}) \right\Vert_{\mathrm L^2(\RR^{d+1})}
\left\Vert \nabla \psi \right\Vert_{\mathrm L^2(\Omega_T)}
\\
& \le C \eps \Big\{
\left\Vert \nabla^2 \widetilde{f_0} \right\Vert_{\mathrm L^2(\RR^{d+1})}
+ \left\Vert \partial_t \widetilde{f_0} \right\Vert_{\mathrm L^2(\RR^{d+1})}
\Big\}
\left\Vert \nabla \psi \right\Vert_{\mathrm L^2(\Omega_T)}
\\
& \le C \eps \Big\{
\left\Vert f_0 \right\Vert_{\mathrm L^2(0,T;\mathrm H^2(\Omega))} 
+ \left\Vert \partial_t f_0 \right\Vert_{\mathrm L^2(\Omega_T)}
\Big\}
\left\Vert \nabla \psi \right\Vert_{\mathrm L^2(\Omega_T)} 
\\
& \le C \eps \Big\{
\left\Vert f_0 \right\Vert_{\mathrm L^2(0,T;\mathrm H^2(\Omega))} 
+ \left\Vert \partial_t f_0 \right\Vert_{\mathrm L^2(0,T;W^{1,d}(\Omega))}
\Big\}
\left\Vert \nabla \psi \right\Vert_{\mathrm L^2(\Omega_T)}, 
\end{aligned}
\end{equation*}
where we have used Lemma \ref{smooth: error} from the Appendix and have employed the H\"older inequality in the last step. Using a property of $S_\eps$ (see Lemma \ref{smooth: extension} from the Appendix), we arrive at the following bound:
\begin{equation*}
\begin{aligned}
J_3 & \le 
C \left\Vert \nabla f_0 \right\Vert_{\mathrm L^2(\Omega_{T,5\sqrt{\eps}})}
\left\Vert \nabla \psi \right\Vert_{\mathrm L^2(\Omega_T)}
\\
& = C \left\{ \left( \int_0^T \int_{\Omega_{5\sqrt{\eps}}} \left\vert \nabla f_0 \right\vert^2 \right)^{\frac12}
+ \left( \int_0^{25\eps} \int_\Omega \left\vert \nabla f_0 \right\vert^2 \right)^{\frac12}
+ \left( \int_{T-25\eps}^T \int_\Omega \left\vert \nabla f_0 \right\vert^2 \right)^{\frac12}
\right\}
\left\Vert \nabla \psi \right\Vert_{\mathrm L^2(\Omega_T)}
\\
& \le C \eps^{\frac14} \Bigg\{ \left\Vert f_0 \right\Vert_{\mathrm L^2(0,T;H^2(\Omega))}
+ \sup_{25\eps < t < T} \left(\frac{1}{5\sqrt{\eps}} \int_{t-25\eps}^t \int_\Omega \left\vert \nabla f_0 \right\vert^2 \right)^{\frac12}
\Bigg\} \left\Vert \nabla \psi \right\Vert_{\mathrm L^2(\Omega_T)} 
\\
& \le C \eps^{\frac14} \left\{ \left\Vert f_0 \right\Vert_{\mathrm L^2(0,T;H^2(\Omega))}
+ \left\Vert \partial_t f_0 \right\Vert_{\mathrm L^2(\Omega_T)}
+ \left\Vert  f_{\text{in}} \right\Vert_{\mathrm L^2(\Omega)}
\right\} \left\Vert \nabla \psi \right\Vert_{\mathrm L^2(\Omega_T)} 
\\
& \le C \eps^{\frac14} \left\{ \left\Vert f_0 \right\Vert_{\mathrm L^2(0,T;H^2(\Omega))}
+ \left\Vert \partial_t f_0 \right\Vert_{\mathrm L^2(0,T;W^{1,d}(\Omega))}
+ \left\Vert  f_{\text{in}} \right\Vert_{\mathrm L^2(\Omega)}
\right\} \left\Vert \nabla \psi \right\Vert_{\mathrm L^2(\Omega_T)} ,
\end{aligned}
\end{equation*}
where we have used the fact \cite[Remark 3.6, page 2106]{geng2017convergence} that for any $g\in\mathrm H^1(\Omega)$, 
\begin{equation*}
\left\Vert g \right\Vert_{\mathrm L^2(\Omega_\delta)} \le C \sqrt{\delta} \left\Vert g \right\Vert_{\mathrm H^1(\Omega)}
\end{equation*}
and the following estimate \cite[Remark 3.10, page 2108]{geng2017convergence}:
\begin{equation*}
\sup_{25\eps < t < T} \left(\frac{1}{5\sqrt{\eps}} \int_{t-25\eps}^t \int_\Omega \left\vert \nabla f_0 \right\vert^2 \right)^{\frac12}  \le C \left\{\left\Vert \partial_t f_0 \right\Vert_{\mathrm L^2(\Omega_T)} + \left\Vert  f_{\text{in}} \right\Vert_{\mathrm L^2(\Omega)} \right\}
\end{equation*}
Putting together the above estimates on $J_3$ and $J_4$ results in the following bound on $I_3$:
\begin{equation*}
I_3 \le C \eps^{\frac14}\left\{
\left\Vert f_0 \right\Vert_{\mathrm L^2(0,T;H^2(\Omega))}
+ \left\Vert \partial_t f_0 \right\Vert_{\mathrm L^2(0,T;W^{1,d}(\Omega))}
+ \left\Vert  f_{\text{in}} \right\Vert_{\mathrm L^2(\Omega)}
\right\} \left\Vert \nabla \psi \right\Vert_{\mathrm L^2(\Omega_T)}.
\end{equation*}
To estimate $I_4$, note that
\begin{equation}\label{eq:grad-on-Seps}
\nabla \bigg\{S_{\varepsilon}(\eta_1^\eps \eta_2^\eps \nabla f_0)\bigg\}
= S_{\varepsilon}\Big(\nabla \left(\eta_1^\eps \eta_2^\eps\right) \nabla f_0\Big)
+ S_{\varepsilon}\Big(\eta_1^\eps \eta_2^\eps \nabla^2 f_0\Big).
\end{equation}
Hence $I_4$ can be rewritten as
\begin{equation*}
\begin{aligned}
I_4 & \le \eps \left\vert \int_{\Omega_T} \Theta_\eps S_{\varepsilon}\Big(\nabla \left(\eta_1^\eps \eta_2^\eps\right) \nabla f_0\Big) \omega_\eps \cdot \nabla \psi\right\vert
+ \eps \left\vert \int_{\Omega_T} \Theta_\eps S_{\varepsilon}\Big(\eta_1^\eps \eta_2^\eps \nabla^2 f_0\Big) \omega_\eps \cdot \nabla \psi\right\vert
\\
& \le C \eps \Bigg\{ 
\left\Vert S_{\varepsilon}\Big(\nabla \left(\eta_1^\eps \eta_2^\eps\right) \nabla f_0\Big) \right\Vert_{\mathrm L^2(\Omega_{T,5\sqrt{\eps}})}
+ \left\Vert S_{\varepsilon}\Big(\eta_1^\eps \eta_2^\eps \nabla^2 f_0\Big) \right\Vert_{\mathrm L^2(\Omega_T)}
\Bigg\} \left\Vert \nabla \psi \right\Vert_{\mathrm L^2(\Omega_T)}
\\
& \le C \eps \Bigg\{
\frac{1}{\sqrt{\eps}} \left\Vert \nabla f_0 \right\Vert_{\mathrm L^2(\Omega_T)}
+ \left\Vert \nabla^2 f_0 \right\Vert_{\mathrm L^2(\Omega_T)}
\Bigg\} \left\Vert \nabla \psi \right\Vert_{\mathrm L^2(\Omega_T)}
\\
& \le C \eps^{\frac12}
\left\Vert f_0 \right\Vert_{\mathrm L^2(0,T; H^2(\Omega))}
\left\Vert \nabla \psi \right\Vert_{\mathrm L^2(\Omega_T)},
\end{aligned}
\end{equation*}
where we have used Lemma \ref{smooth: compact support} from the Appendix. Using the property \eqref{eq:grad-on-Seps}, and arguing as we did in the case of $I_4$, we can arrive at the following bound for $I_5$:
\begin{equation*}
I_5 \le C \eps^{\frac12}
\left\Vert f_0 \right\Vert_{\mathrm L^2(0,T; H^2(\Omega))}
\left\Vert \nabla \psi \right\Vert_{\mathrm L^2(\Omega_T)}
\end{equation*}
Collecting all the bounds derived above, we arrive at \eqref{eq:wkform-bound}. Hence the proof.
\end{proof}
We are now equipped to prove Theorem \ref{thm:homogen-feps}.
\begin{proof}[Proof of Theorem \ref{thm:homogen-feps}]
As $w_\eps\in \mathrm L^2(0,T;\mathrm H^1_0(\Omega))$, we take $\psi=w_\eps$ in the estimate \eqref{eq:wkform-bound} obtained in the previous lemma to arrive at the following inequality:
\begin{equation*}
\begin{aligned}
\int_0^T & \left\langle \zeta_\eps\partial_{t}w_{\varepsilon}, w_\eps \right\rangle_{\mathrm{H}^{-1}(\Omega) \times \mathrm{H}_0^1(\Omega)} 
+ \int_{\Omega_T} \Theta_\eps\nabla w_{\varepsilon} \cdot \nabla w_\eps 
\\
& \le C \eps^{\frac14} \left\{ \left\Vert \partial_tf_0\right\Vert_{L^2(0,T;W^{1,d}(\Omega))} + \left\Vert f_0\right\Vert_{\mathrm L^2(0,T;\mathrm H^2(\Omega))} + \left\Vert f_{{\rm{in}}}\right\Vert_{L^2(\Omega)} \right\}
\left\Vert w_\eps\right\Vert_{\mathrm L^2(0,T;\mathrm H^1_0(\Omega))}.
\end{aligned}
\end{equation*}
Using the coercivity assumption on $\Theta_\eps$, the left hand side of the above inequality can be bounded from below by 
\begin{equation*}
\int_0^T \frac{{\rm d}}{{\rm d}t} \left\Vert \sqrt{\zeta_\eps} w_\eps \right\Vert^2_{\mathrm L^2(\Omega)}
+ \kappa  \left\Vert \nabla w_\eps \right\Vert^2_{\mathrm L^2(\Omega_T)}
= \left\Vert \sqrt{\zeta_\eps} w_\eps(T,\cdot) \right\Vert^2_{\mathrm L^2(\Omega)}
+ \kappa  \left\Vert \nabla w_\eps \right\Vert^2_{\mathrm L^2(\Omega_T)}.
\end{equation*}
This helps us deduce
\begin{equation*}
\begin{aligned}
\left\Vert w_\eps\right\Vert_{\mathrm L^2(0,T;\mathrm H^1_0(\Omega))}
& \le C \eps^{\frac14} \left\{ \left\Vert \partial_tf_0\right\Vert_{L^2(0,T;W^{1,d}(\Omega))} + \left\Vert f_0\right\Vert_{\mathrm L^2(0,T;\mathrm H^2(\Omega))} + \left\Vert f_{{\rm{in}}}\right\Vert_{L^2(\Omega)} \right\}.
\end{aligned}
\end{equation*}
Thanks to Poincar\'e inequality, we have
\begin{equation*}
\left\Vert w_\eps\right\Vert_{\mathrm L^2(\Omega_T)}
\le C
\left\Vert w_\eps\right\Vert_{\mathrm L^2(0,T;\mathrm H^1_0(\Omega))}
\end{equation*}
Using the definition \eqref{eq:defn-weps} of $w_\eps$, we have
\begin{equation*}
\begin{aligned}
\left\Vert f_\eps - f_0 \right\Vert_{\mathrm L^2(\Omega_T)} 
& \le \left\Vert w_\eps\right\Vert_{\mathrm L^2(\Omega_T)}
+ \eps \left\Vert \omega\left(\frac{x}{\eps}\right)\cdot S_\eps \left( \eta_1^\eps \eta_2^\eps \nabla f_0 \right) \right\Vert_{\mathrm L^2(\Omega_T)}
\\
& \le \left\Vert w_\eps\right\Vert_{\mathrm L^2(\Omega_T)}
+ \eps \left\Vert \nabla f_0 \right\Vert_{\mathrm L^2(\Omega_T)}
\end{aligned}
\end{equation*}
Hence the estimate follows.
\end{proof}
\begin{remark}\label{rem:improved-rate}
By optimizing on the supports of the cut-off functions used in the above proof, we can improve upon the rate obtained in \eqref{eq:L2-estimate}. More precisely, if we were to take the space cut-off function to be
\[
\eta_1^\eps = 1 \quad \mbox{ in }\quad \Omega \backslash \Omega_{4{\varepsilon}^{ 4/7}},
\qquad
\eta_1^\eps = 0 \quad \mbox{ in }\quad \Omega_{3{\varepsilon}^{ 4/7}},
\qquad
\left\vert\nabla \eta_1^\eps \right\vert \leq \frac{C}{{\eps}^{ 4/7}} \quad \mbox{ in }\quad \Omega,
\]
and the time cut-off function to be
\[
\eta_2^\eps = 1 \quad \mbox{ in }\, \left(8\eps^{ 8/7}, T-8\eps^{ 8/7}\right),
\qquad
\eta_2^\eps = 0 \quad \mbox{ in }\, (0,4\eps^{ 8/7}]\cup(T-4\eps^{ 8/7},T),
\qquad 
\left\vert \partial_t \eta_2^\eps\right\vert \leq \frac{C}{\eps^{ 8/7}} \quad \mbox{ in }\, (0,T),
\]
and repeating the above arguments will lead to the improved estimate:
\begin{equation*}
\begin{aligned}
\left \Vert f_\eps - f_0 \right\Vert_{\mathrm L^2(\Omega_T)}  
&
\le C \eps^{\frac 27} \left\{ \left\Vert \partial_tf_0\right\Vert_{L^2(0,T;W^{1,d}(\Omega))} + \left\Vert f_0\right\Vert_{\mathrm L^2(0,T;\mathrm H^2(\Omega))} + \left\Vert f_{{\rm{in}}}\right\Vert_{L^2(\Omega)} \right\}. 
\end{aligned}
\end{equation*}
\end{remark}
\section{The case of regular coefficients}\label{sec:3}
In this section, the rate obtained in Theorem \ref{thm:homogen-feps} is improved by imposing an additional regularity assumption on the coefficient $\zeta$.
\begin{theorem}\label{thm:homogen-feps better}
Suppose $\zeta \in \mathrm H_{per}^1(Y) \cap \mathrm L_{per}^{\infty}(Y)$ satisfies \eqref{eq:zeta-average-1}.
Let $f_{\eps}$ be the solution to \eqref{eq:equation-for-feps} and let $f_0$ be the solution to the corresponding homogenized problem \eqref{eq:homogeneqn-f0}. Then there exists a positive constant $C$, independent of $\varepsilon$, such that the following estimate holds:
\begin{equation}\label{eq:L2-estimate better}
\begin{aligned}
\left\Vert f_\eps - f_0 \right\Vert_{\mathrm L^2(\Omega_T)} 
&  \le C \eps^{\frac12} \left\{ \left\Vert \partial_tf_0\right\Vert_{L^2(0,T;W^{1,d}(\Omega))} + \left\Vert f_0\right\Vert_{\mathrm L^2(0,T;\mathrm H^2(\Omega))} + \left\Vert f_{{\rm{in}}}\right\Vert_{L^2(\Omega)} \right\}.
\end{aligned}
\end{equation}
\end{theorem}
Let $P:Y\to\RR^{d\times d}$ be defined as
\begin{equation}\label{eq:matrix-P}
P := \Theta(y) + \Theta(y)\nabla_y\omega(y) - \Theta_h \qquad \mbox{ in }Y.
\end{equation}
Note that $P$ is $Y$-periodic. We shall introduce a new time scale $\tau=\eps^{-2}t$ and consider the following $(d+1)\times d$ sized matrix $H$ defined as follows:
\begin{equation}
H_{ij} := 
\left\{
\begin{aligned}
P_{ij} & \qquad \mbox{ for }1\le i,j\le d,
\\
-\zeta(y)\omega_j(y) & \qquad \mbox{ for }i=d+1 \mbox{ and }1\le j\le d.
\end{aligned}
\right.
\end{equation}
Observe that 
\begin{equation*}
\sum_{i=1}^d \frac{\partial}{\partial y_i}H_{ij} + \frac{\partial}{\partial\tau}H_{(d+1)j} = 0 \qquad \mbox{ for }j=1,\dots,d,
\end{equation*}
and
\begin{equation*}
\int_0^1 \int_Y H_{ij}(y) \, {\rm d}y \, {\rm d}\tau = 0 \qquad \mbox{ for }1\le i\le d+1,\quad 1\le j\le d. 
\end{equation*}
Hence there exists a third order tensor $\mathfrak{B}$ (see \cite[Proposition 3.1.1, page 35]{shen2018periodic}) such that
\begin{equation}\label{eq:3rd-order-tensor}
\frac{\partial \mathfrak{B}_{(d+1)ij}}{\partial \tau} + \sum_{k=1}^d \frac{\partial \mathfrak{B}_{kij}}{\partial y_k} = H_{ij} \qquad \mbox{ for }1\le i\le d+1, \quad 1\le j\le d,
\end{equation}
and
\begin{equation}\label{eq:B-skewsym}
\mathfrak{B}_{kij} = -\mathfrak{B}_{ikj} \qquad \mbox{ for }1\le i,k\le d+1, \qquad 1\le j\le d.
\end{equation}
Similar to the definition of $w_\eps$ in \eqref{eq:defn-weps}, we define $z_\eps:\Omega_T\to\RR$ as follows:
\begin{equation}\label{eq:defn-zeps}
z_\eps := f_\eps - f_0 
- \eps \omega\left(\frac{x}{\eps}\right)\cdot S_\eps \left( \eta_3^\eps \eta_4^\eps \nabla f_0 \right)
- \eps^2 \sum_{i,j=1}^d \frac{\mathfrak{B}_{(d+1)ij}\xeps}{\zeta\xeps} \frac{\partial}{\partial x_i} \left(S_\eps \left( \eta_3^\eps \eta_4^\eps \frac{\partial f_0}{\partial x_j}\right)\right),
\end{equation}
where $\eta_3^\eps\in\mathrm{C}_c^{\infty}(\Omega;[0,1])$ and $\eta_4^\eps\in \mathrm{C}_c^\infty(0,T;[0,1])$ are the space and time cut-off functions that satisfy
\[
\eta_3^\eps = 1 \quad \mbox{ in }\quad \Omega \backslash \Omega_{4\varepsilon},
\qquad
\eta_3^\eps = 0 \quad \mbox{ in }\quad \Omega_{3\varepsilon},
\qquad
\left\vert\nabla \eta_3^\eps \right\vert \leq \frac{1}{\eps} \quad \mbox{ in }\quad \Omega.
\]
\[
\eta_4^\eps = 1 \quad \mbox{ in }\, \left(8\eps^2, T-8\eps^2\right),
\qquad
\eta_4^\eps = 0 \quad \mbox{ in }\, (0,4\eps^2]\cup(T-4\eps^2,T),
\qquad 
\left\vert \partial_t \eta_4^\eps\right\vert \leq \frac{1}{\eps^2} \quad \mbox{ in }\, (0,T).
\]
Note that $z_\eps\in\mathrm L^2(0,T;\mathrm H^1_0(\Omega))$ and $z_\eps(0,\cdot)\equiv0$. Further, a direct computation leads to
\begin{equation*}
\begin{aligned}
\zeta_\eps\partial_{t}z_{\varepsilon}
= \zeta_\eps \partial_t f_\eps 
- \zeta_\eps \partial_t f_0 
& - \eps \zeta_\eps \omega_\eps\cdot \partial_t \left\{S_\eps \left( \eta_3^\eps \eta_4^\eps \nabla f_0 \right)\right\}
\\ 
& - \eps^2 \frac{\partial}{\partial t} \left( \sum_{i,j=1}^d \mathfrak{B}_{(d+1)ij}\xeps \frac{\partial}{\partial x_i} \left(S_\eps \left( \eta_3^\eps \eta_4^\eps \frac{\partial f_0}{\partial x_j}\right)\right)\right)
\end{aligned}
\end{equation*}
and
\begin{equation*}
\begin{aligned}
\Theta_\eps \nabla z_\eps 
& = \Theta_\eps \nabla f_\eps
- \Theta_h \nabla f_0 
+ \left(\Theta_h - \Theta_\eps\right) \left( \nabla f_0 - S_\eps \left( \eta_3^\eps \eta_4^\eps \nabla f_0 \right) \right)
- P\xeps S_\eps \left( \eta_3^\eps \eta_4^\eps \nabla f_0 \right)
\\
& \quad 
- \eps \Theta_\eps \nabla \left\{S_\eps \left( \eta_3^\eps \eta_4^\eps \nabla f_0 \right)\right\} \omega_\eps
- \eps^2 \Theta_\eps\nabla \left(\sum_{i,j=1}^d \frac{\mathfrak{B}_{(d+1)ij}\xeps}{\zeta\xeps} \frac{\partial}{\partial x_i} \left(S_\eps \left( \eta_3^\eps \eta_4^\eps \frac{\partial f_0}{\partial x_j}\right)\right)\right).
\end{aligned}
\end{equation*}
Therefore for any $\psi\in\mathrm L^2(0,T;\mathrm H^1_0(\Omega))$, we get
\begin{equation}\label{eq:wk-form-1}
\begin{aligned}
\int_0^T & \left\langle \zeta_\eps\partial_{t}z_{\varepsilon},\psi \right\rangle_{\mathrm{H}^{-1}(\Omega) \times \mathrm{H}_0^1(\Omega)} 
+ \int_{\Omega_T} \Theta_\eps\nabla z_{\varepsilon} \cdot \nabla \psi
= - \int_{\Omega_T} (\zeta_\eps -1)\partial_{t}f_{0}\psi
\\
& - \eps \int_{\Omega_T} \zeta_\eps \omega_\eps\cdot \partial_t \left\{S_\eps \left( \eta_3^\eps \eta_4^\eps \nabla f_0 \right)\right\} \psi
+ \int_{\Omega_T}\left(\Theta_h - \Theta_\eps\right) \left( \nabla f_0 - S_\eps \left( \eta_3^\eps \eta_4^\eps \nabla f_0 \right) \right) \cdot \nabla \psi
\\
& - \eps^2 \int_{\Omega_T} \frac{\partial}{\partial t} \left( \sum_{i,j=1}^d \mathfrak{B}_{(d+1)ij}\xeps \frac{\partial}{\partial x_i} \left(S_\eps \left( \eta_3^\eps \eta_4^\eps \frac{\partial f_0}{\partial x_j}\right)\right)\right) \psi
\\
& + \int_{\Omega_T} \nabla \cdot \left( P\xeps S_\eps \left( \eta_3^\eps \eta_4^\eps \nabla f_0 \right) \right) \psi
- \eps \int_{\Omega_T} \Theta_\eps \nabla \left\{S_\eps \left( \eta_3^\eps \eta_4^\eps \nabla f_0 \right)\right\} \omega_\eps \cdot \nabla \psi
\\
&- \eps^2 \int_{\Omega_T} \Theta_\eps\nabla \left(\sum_{i,j=1}^d \frac{\mathfrak{B}_{(d+1)ij}\xeps}{\zeta\xeps} \frac{\partial}{\partial x_i} \left(S_\eps \left( \eta_3^\eps \eta_4^\eps \frac{\partial f_0}{\partial x_j}\right)\right)\right) \cdot \nabla \psi
\end{aligned}
\end{equation}
Observe that, thanks to $P$ (defined by \eqref{eq:matrix-P}) having divergence-free columns, we have
\begin{equation*}
\begin{aligned}
& -\eps \sum_{j=1}^d \zeta\xeps\omega_j\xeps\partial_t \left\{S_\eps \left( \eta_3^\eps \eta_4^\eps \frac{\partial f_0}{\partial x_j} \right)\right\}
+ \sum_{i,j=1}^d \frac{\partial}{\partial x_i} \left(P_{ij}\xeps S_\eps \left( \eta_3^\eps \eta_4^\eps \frac{\partial f_0}{\partial x_j} \right) \right)
\\
& = - \eps \sum_{j=1}^d \zeta\xeps\omega_j\xeps\partial_t \left\{S_\eps \left( \eta_3^\eps \eta_4^\eps \frac{\partial f_0}{\partial x_j} \right)\right\}
+ \sum_{i,j=1}^d P_{ij}\xeps \frac{\partial}{\partial x_i} \left( S_\eps \left( \eta_3^\eps \eta_4^\eps \frac{\partial f_0}{\partial x_j} \right) \right)
\end{aligned}
\end{equation*}
Employing the third-order skew-symmetric tensor mentioned in \eqref{eq:3rd-order-tensor}, the expression on the right hand side of the above equality becomes
\begin{equation*}
\begin{aligned}
& \eps^2 \sum_{j,k=1}^d \frac{\partial}{\partial x_k} \left( \mathfrak{B}_{k(d+1)j}\xeps\right) \partial_t \left\{S_\eps \left( \eta_3^\eps \eta_4^\eps \frac{\partial f_0}{\partial x_j} \right)\right\}
+ \eps \sum_{i,j,k=1}^d \frac{\partial}{\partial x_k} \left( \mathfrak{B}_{kij}\xeps\right) \frac{\partial}{\partial x_i} \left( S_\eps \left( \eta_3^\eps \eta_4^\eps \frac{\partial f_0}{\partial x_j} \right) \right)
\\
& + \eps^3 \sum_{j=1}^d \frac{\partial}{\partial t} \left( \mathfrak{B}_{(d+1)(d+1)j}\xeps\right) \partial_t \left\{S_\eps \left( \eta_3^\eps \eta_4^\eps \frac{\partial f_0}{\partial x_j} \right)\right\}
+ \eps^2 \sum_{i,j=1}^d  \frac{\partial}{\partial t} \left( \mathfrak{B}_{(d+1)ij}\xeps\right) \frac{\partial}{\partial x_i} \left( S_\eps \left( \eta_3^\eps \eta_4^\eps \frac{\partial f_0}{\partial x_j} \right)\right)
\end{aligned}
\end{equation*}
As the third-order tensor $\mathfrak{B}$ is independent of the $t$ variable, the last two terms of the above expressions should vanish. But, in view of an algebraic manipulations that aids us next, we choose to ignore the apparent $\mathcal{O}(\eps^3)$ term while still retaining the last term. Furthermore, as a simple consequence of the product rule for differentiation, the above expression can be rewritten as
\begin{equation*}
\begin{aligned}
& \eps^2 \sum_{j,k=1}^d \frac{\partial}{\partial x_k} \left( \mathfrak{B}_{k(d+1)j}\xeps \partial_t \left\{S_\eps \left( \eta_3^\eps \eta_4^\eps \frac{\partial f_0}{\partial x_j} \right)\right\}\right)
- \eps^2 \sum_{j,k=1}^d \mathfrak{B}_{k(d+1)j}\xeps \frac{\partial^2}{\partial x_k\partial t}\left\{S_\eps \left( \eta_3^\eps \eta_4^\eps \frac{\partial f_0}{\partial x_j} \right)\right\}
\\
& + \eps \sum_{i,j,k=1}^d \frac{\partial}{\partial x_k} \left( \mathfrak{B}_{kij}\xeps \frac{\partial}{\partial x_i} \left( S_\eps \left( \eta_3^\eps \eta_4^\eps \frac{\partial f_0}{\partial x_j} \right)\right)\right)
- \eps \sum_{i,j,k=1}^d \mathfrak{B}_{kij}\xeps \frac{\partial^2}{\partial x_k\partial x_i} \left\{ S_\eps \left( \eta_3^\eps \eta_4^\eps \frac{\partial f_0}{\partial x_j} \right) \right\}
\\
& + \eps^2 \sum_{i,j=1}^d  \frac{\partial}{\partial t} \left( \mathfrak{B}_{(d+1)ij}\xeps \frac{\partial}{\partial x_i} \left( S_\eps \left( \eta_3^\eps \eta_4^\eps \frac{\partial f_0}{\partial x_j} \right)\right)\right)
- \eps^2 \sum_{i,j=1}^d \mathfrak{B}_{(d+1)ij}\xeps \frac{\partial^2}{\partial x_i\partial t} \left\{ S_\eps \left( \eta_3^\eps \eta_4^\eps \frac{\partial f_0}{\partial x_j} \right) \right\},
\end{aligned}
\end{equation*}
Note further that, owing to the skew-symmetry of $\mathfrak{B}$ (see \eqref{eq:B-skewsym}), the second, the fourth and the sixth terms in the final expression together amount to zero. 

We have thus proved the following result:
\begin{lemma}\label{lemma: higher regular time derivative}
Let $z_\eps$ be defined by \eqref{eq:defn-zeps}. Then for any $\psi \in \mathrm{L}^2(0,T;\mathrm H_0^1(\Omega))$, we have
\begin{equation}\label{eq:wk-form-2}
\begin{aligned}
\int_0^T & \left\langle \zeta_\eps\partial_{t}z_{\varepsilon},\psi \right\rangle_{\mathrm{H}^{-1}(\Omega) \times \mathrm{H}_0^1(\Omega)} 
+ \int_{\Omega_T} \Theta_\eps\nabla z_{\varepsilon} \cdot \nabla \psi
= - \int_{\Omega_T} (\zeta_\eps -1)\partial_{t}f_{0}\psi
\\
& 
+ \int_{\Omega_T}\left(\Theta_h - \Theta_\eps\right) \left( \nabla f_0 - S_\eps \left( \eta_3^\eps \eta_4^\eps \nabla f_0 \right) \right) \cdot \nabla \psi
- \eps \int_{\Omega_T} \Theta_\eps \nabla \left\{S_\eps \left( \eta_3^\eps \eta_4^\eps \nabla f_0 \right)\right\} \omega_\eps \cdot \nabla \psi
\\
&- \eps^2 \int_{\Omega_T} \Theta_\eps\nabla \left(\sum_{i,j=1}^d \frac{\mathfrak{B}_{(d+1)ij}\xeps}{\zeta\xeps} \frac{\partial}{\partial x_i} \left(S_\eps \left( \eta_3^\eps \eta_4^\eps \frac{\partial f_0}{\partial x_j}\right)\right)\right) \cdot \nabla \psi
\\
& - \eps^2 \sum_{j,k=1}^d \int_{\Omega_T} \mathfrak{B}_{k(d+1)j}\xeps \partial_t \left\{S_\eps \left( \eta_3^\eps \eta_4^\eps \frac{\partial f_0}{\partial x_j} \right)\right\} \frac{\partial\psi}{\partial x_k}
\\
& - \eps \sum_{i,j,k=1}^d \int_{\Omega_T} \mathfrak{B}_{kij}\xeps \frac{\partial}{\partial x_i} \left( S_\eps \left( \eta_3^\eps \eta_4^\eps \frac{\partial f_0}{\partial x_j} \right)\right) \frac{\partial\psi}{\partial x_k}.
\end{aligned}
\end{equation}
\end{lemma}
The next result aims to get a bound on the left hand side of \eqref{eq:wk-form-2}.
\begin{lemma}\label{lem:bound-weak-form-2}
Let $z_\eps$ be defined by \eqref{eq:defn-zeps}. Then for any $\psi \in \mathrm{L}^2(0,T;\mathrm H_0^1(\Omega))$, we have
\begin{equation*}
\begin{aligned}
\int_0^T & \left\langle \zeta_\eps\partial_{t}z_{\varepsilon},\psi \right\rangle_{\mathrm{H}^{-1}(\Omega) \times \mathrm{H}_0^1(\Omega)} 
+ \int_{\Omega_T} \Theta_\eps\nabla z_{\varepsilon} \cdot \nabla \psi
\\
& \le C \eps^{\frac12} \left\{ \left\Vert \partial_tf_0\right\Vert_{L^2(0,T;W^{1,d}(\Omega))} + \left\Vert f_0\right\Vert_{\mathrm L^2(0,T;\mathrm H^2(\Omega))} + \left\Vert f_{{\rm{in}}}\right\Vert_{L^2(\Omega)} \right\} \left\Vert \psi \right\Vert_{\mathrm L^2(0,T;\mathrm H^1_0(\Omega))}.
\end{aligned}
\end{equation*}
\end{lemma} 
\begin{proof}
Thanks to the expression in \eqref{eq:wk-form-2}, we obtain 
\begin{equation*}
\begin{aligned}
\int_0^T & \left\langle \zeta_\eps\partial_{t}z_{\varepsilon},\psi \right\rangle_{\mathrm{H}^{-1}(\Omega) \times \mathrm{H}_0^1(\Omega)} 
+ \int_{\Omega_T} \Theta_\eps\nabla z_{\varepsilon} \cdot \nabla \psi
\le \left\vert \int_{\Omega_T} (\zeta_\eps -1)\partial_{t}f_{0}\psi \right\vert
\\
& 
+ \left\vert \int_{\Omega_T}\left(\Theta_h - \Theta_\eps\right) \left( \nabla f_0 - S_\eps \left( \eta_3^\eps \eta_4^\eps \nabla f_0 \right) \right) \cdot \nabla \psi \right\vert
+ \eps \left\vert \int_{\Omega_T} \Theta_\eps \nabla \left\{S_\eps \left( \eta_3^\eps \eta_4^\eps \nabla f_0 \right)\right\} \omega_\eps \cdot \nabla \psi \right\vert
\\
& + \eps^2 \left\vert \int_{\Omega_T} \Theta_\eps\nabla \left(\sum_{i,j=1}^d \frac{\mathfrak{B}_{(d+1)ij}\xeps}{\zeta\xeps} \frac{\partial}{\partial x_i} \left(S_\eps \left( \eta_3^\eps \eta_4^\eps \frac{\partial f_0}{\partial x_j}\right)\right)\right) \cdot \nabla \psi \right\vert
\\
& + \eps^2 \left\vert \sum_{j,k=1}^d \int_{\Omega_T} \mathfrak{B}_{k(d+1)j}\xeps \partial_t \left\{S_\eps \left( \eta_3^\eps \eta_4^\eps \frac{\partial f_0}{\partial x_j} \right)\right\} \frac{\partial\psi}{\partial x_k} \right\vert
\\
& + \eps \left\vert \sum_{i,j,k=1}^d \int_{\Omega_T} \mathfrak{B}_{kij}\xeps \frac{\partial}{\partial x_i} \left( S_\eps \left( \eta_3^\eps \eta_4^\eps \frac{\partial f_0}{\partial x_j} \right)\right) \frac{\partial\psi}{\partial x_k} \right\vert =: \ell_1 + \ell_2 + \ell_3 + \ell_4 + \ell_5 + \ell_6.
\end{aligned}
\end{equation*}
Note that $\ell_1$ is the same as $I_1$ in the proof of Lemma \ref{lem:wkform-bound}. Hence we have
\begin{equation}\label{eq:bd-ell-1}
\ell_1 \le C \eps\left\Vert \partial_t f_0 \right\Vert_{L^2(0,T;\mathrm W^{1,d}(\Omega))} \left\Vert \psi\right\Vert_{\mathrm L^2(0,T;\mathrm H^1_0(\Omega))} .
\end{equation}
Also note that $\ell_2$ is similar to $I_3$ in the proof of Lemma \ref{lem:wkform-bound}, except for the choice of the cut-off functions. Mimicking the arguments that dealt with $I_3$ yields the following bound:
\begin{equation}\label{eq:bd-ell-2}
\ell_2 \le C \eps^{\frac12}\left\{
\left\Vert f_0 \right\Vert_{\mathrm L^2(0,T;H^2(\Omega))}
+ \left\Vert \partial_t f_0 \right\Vert_{\mathrm L^2(0,T;W^{1,d}(\Omega))}
+ \left\Vert  f_{\text{in}} \right\Vert_{\mathrm L^2(\Omega)}
\right\} \left\Vert \nabla \psi \right\Vert_{\mathrm L^2(\Omega_T)}.
\end{equation}
Further, $\ell_3$ is similar to $I_4$ in the proof of Lemma \ref{lem:wkform-bound} (again, except for the choice of the cut-off functions). A repeat of those arguments results in
\begin{equation}\label{eq:bd-ell-3}
\ell_3 \le C \eps^{\frac12}
\left\Vert f_0\right\Vert_{\mathrm L^2(0,T;\mathrm H^2(\Omega))}
\left\Vert \nabla \psi \right\Vert_{\mathrm L^2(\Omega_T)}.
\end{equation}
Note that
\begin{equation}\label{eq:inter-1}
\frac{\partial}{\partial x_i} \left(S_\eps \left( \eta_3^\eps \eta_4^\eps \frac{\partial f_0}{\partial x_j}\right)\right)
= S_\eps \left( \frac{\partial}{\partial x_i}\left(\eta_3^\eps \eta_4^\eps\right) \frac{\partial f_0}{\partial x_j}\right)
+ S_\eps \left( \eta_3^\eps \eta_4^\eps \frac{\partial^2 f_0}{\partial x_i \partial x_j}\right)
\end{equation}
and
\begin{equation}\label{eq:inter-2}
\nabla^2 S_\eps\left(\eta_3^\eps \eta_4^\eps \nabla f_0\right) 
= \nabla S_\eps\left(\nabla\left(\eta^\eps_3 \eta^\eps_4\right) \nabla f_0\right) 
+ \nabla S_\eps\left(\eta^\eps_3 \eta^\eps_4 \nabla^2 f_0\right)
\end{equation}
Considering $\ell_4$, we observe that
\begin{equation*}
\begin{aligned}
\ell_4 & = \eps \left\vert \int_{\Omega_T} \Theta_\eps \left(\sum_{i,j=1}^d \nabla \left(\frac{\mathfrak{B}_{(d+1)ij}}{\zeta}\right)\xeps \frac{\partial}{\partial x_i} \left(S_\eps \left( \eta_3^\eps \eta_4^\eps \frac{\partial f_0}{\partial x_j}\right)\right)\right) \cdot \nabla \psi \right\vert
\\
& \qquad + \eps^2 \left\vert \int_{\Omega_T} \Theta_\eps \left(\sum_{i,j=1}^d \frac{\mathfrak{B}_{(d+1)ij}\xeps}{\zeta\xeps} \nabla \left( \frac{\partial}{\partial x_i} \left(S_\eps \left( \eta_3^\eps \eta_4^\eps \frac{\partial f_0}{\partial x_j}\right)\right)\right)\right) \cdot \nabla \psi \right\vert.
\end{aligned}
\end{equation*}
Using \eqref{eq:inter-1} and \eqref{eq:inter-2} and a few properties of the smoothing operator $S_\eps$ (see Lemma \ref{smooth: compact support} and \ref{smooth: extension} in the Appendix), we obtain the bound:
\begin{equation}\label{eq:bd-ell-4}
\begin{aligned}
\ell_4 & \leq C  \Vert \nabla f_0 \Vert_{\mathrm L^2(\Omega_{T,3\eps})} \Vert \nabla \psi \Vert_{\mathrm L^2(\Omega_{T,3\eps})} 
+ C \eps \Vert \nabla^2 f_0 \Vert_{\mathrm L^2(\Omega_T)} \Vert \nabla \psi \Vert_{\mathrm L^2(\Omega_T)}
\\
& \le C \eps^\frac12 \left\Vert \nabla^2 f_0 \right\Vert_{\mathrm L^2(\Omega_T)} \Vert \nabla \psi \Vert_{\mathrm L^2(\Omega_T)}.
\end{aligned}
\end{equation}
Note that to arrive at the estimate \eqref{eq:bd-ell-4}, we have used the fact that $\frac{\mathfrak{B}_{(d+1)ij}}{\zeta}\in \mathrm H^1(\Omega)$ -- see Lemma \ref{lemma: quotient} in the Appendix. We emphasise here that it is at this point we require the coefficient $\zeta$ to be more regular. The terms $\ell_5$ and $\ell_6$ are similar to $I_5$ in the proof of Lemma \ref{lem:wkform-bound} (again, except for the choice of the cut-off functions) and hence can be handled in a similar fashion while taking into account the identity:
\begin{equation*}
\partial_t \left( S_\eps\left(\eta^\eps_3 \eta^\eps_4 \nabla f_0\right) \right) 
= S_{\eps}(\partial_t(\eta^\eps_3 \eta^\eps_4) \nabla f_0) 
+ \nabla S_{\eps}\left(\eta^\eps_3 \eta^\eps_4\partial_t f_0\right) - S_\eps\left(\nabla \left(\eta^\eps_3 \eta^\eps_4\right) \partial_t f_0\right)
\end{equation*}
As the arguments leading to the rest of the  desire estimates are a repeat of those made earlier, we are skipping the details.
\end{proof}

Finally, the proof of Theorem \ref{thm:homogen-feps better} goes along similar lines as the proof of Theorem \ref{thm:homogen-feps}. Essentially, it amounts to taking $\psi=z_\eps$ in the inequality proved in Lemma \ref{lem:bound-weak-form-2} and argue using Poincar\'e inequality. Again, we are choosing not to give the details in the interest of space.

\section{Concluding remarks}\label{sec:conclude}
We finish this manuscript by commenting on the $\eps$-decay rate \eqref{eq:decay-ueps} for the solution $u_\eps$ to our original model problem \eqref{eq:main} (see Remark \ref{rem:decay-ueps}). As the calculations in Remark \ref{rem:decay-ueps} advocate, the decay rate \eqref{eq:decay-ueps} is guaranteed by the strict positivity of the eigenvalue $\lambda$. On the other hand, if $\lambda=0$, then we have
\begin{equation}\label{eq:u-eps-lambda-0}
u_\eps \to v_0 (t,x) \int_Y \psi(y)\, {\rm d}y \quad \mbox{ weakly in }\mathrm L^2(\Omega_T),
\end{equation}
as $\eps\to0$, where $v_0$ solves the homogenized equation \eqref{eq:homogeneqn-v0} (essentially, no decay in $\eps$). Taking cues from \cite[Remark 2.2, page 737]{allaire2007homogenization}, for the case of $A$ being symmetric and $b=0$, we have $\theta=0$. That is, $\psi$ solves the following $Y$-periodic eigenvalue problem:
\[
-\dv_y ({A}(y)\nabla \psi(y)) + c(y) \psi(y) = \lambda \psi(y) \qquad \mbox{ in }Y.
\]
Here we distinguish two cases:
\begin{itemize}
\item $c(y)\ge \delta >0$: Here we deduce that $\lambda\ge \delta$. Hence the $\mathcal{O}(\eps)$ decay rate \eqref{eq:decay-ueps} holds true.
\item $c(y)\equiv0$: Here we have $\lambda=0$. This leads to the behaviour \eqref{eq:u-eps-lambda-0} as $\eps\to0$.
\end{itemize}
We emphasise here that \cite[Remark 2.2, page 737]{allaire2007homogenization} gives another sufficient condition on the coefficients (so-called central symmetry in $Y$) to result in $\theta=0$, it is unclear whether $\lambda=0$ in such scenarios where the coefficients $b$ and $c$ are nontrivial. It is worth exploring these finer aspects of the $\theta$-exponential eigenvalue problems. 

Finally, we comment on the rates of convergence obtained in this work. The $\mathcal{O}(\eps^\frac27)$ rate for $\left \Vert f_\eps - f_0 \right\Vert_{\mathrm L^2(\Omega_T)}$ is the best one can obtain (see Remark \ref{rem:improved-rate}) using the methodology employed in Section \ref{sec:2} under the assumption that the coefficient $\zeta\in\mathrm L^\infty_{\text{per}}(Y)$. The $\mathcal{O}(\eps^\frac12)$ obtained in Theorem \ref{thm:homogen-feps better} heavily relies on the regularity assumption made on $\zeta$. It would be interesting to improve the rates in Section \ref{sec:2} without requiring weak differentiability of the coefficients. 


\appendix
\section{Few Useful Results}

\begin{lemma}\label{smooth: compact support}
For any function $\phi \in \mathrm{L}^{2}(\Omega_T)$ with $\supp{\phi} \subset \subset \Omega_T$, we have the following
\begin{enumerate}
\item $|| S_{\eps}(\phi)||_{\mathrm{L}^2(\Omega_T)} \leq C ||\phi||_{\mathrm{L}^2(\Omega_T)}$
\item $\eps || \nabla S_{\eps}(\phi)||_{\mathrm{L}^2(\Omega_T)} + \eps^2 || \nabla^2 S_{\eps}(\phi)||_{\mathrm{L}^2(\Omega_T)} \leq C ||\phi||_{\mathrm{L}^2(\Omega_T)}$
\end{enumerate}
where the constant $C$ depends upon dimension $d$.
\end{lemma}

\begin{lemma}\label{smooth: error}
For any function $f \in \mathrm{L}^{2}(\RR;\mathrm{H}^{2}(\RR^d))$, we have the following
\[
|| \nabla S_{\eps}(f) - \nabla f ||_{\mathrm{L}^2(\RR^{d+1})} \leq C \eps \left\{ || \nabla^2 f ||_{\mathrm{L}^2(\RR^{d+1})} + || \partial_t f ||_{\mathrm{L}^2(\RR^{d+1})} \right\}
\]
where the constant $C$ depends upon dimension $d$.
\end{lemma}

\begin{lemma}\label{smooth: extension}
For $Y$-periodic function $g$ and $f \in \mathrm{L}^{p}(\Omega_T)$ consider any subset $\omega \subset \subset \Omega$ and $0<s<\tau<T$ then we have the following
\begin{enumerate}
\item $\left\Vert g\xeps S_{\eps}(f) \right\Vert_{\mathrm{L}^p((s,\tau)\times \omega)} \leq C ||g||_{\mathrm{L}^p(Y)} || f ||_{\mathrm{L}^p((s-\eps^2, \tau + \eps^2)\times\omega(\eps))}$
\item $\eps\left\Vert g\xeps \nabla S_{\eps}(f) \right\Vert_{\mathrm{L}^p((s,\tau)\times \omega)} \leq C ||g||_{\mathrm{L}^p(Y)} || f ||_{\mathrm{L}^p((s-\eps^2, \tau + \eps^2)\times\omega(\eps))}$
\end{enumerate}
where $\omega(\eps):=\{x\in \Omega: \dist(x,\omega)<\eps\}$, $0<\eps <\min\{s,\tau\}$ and the constant $C$ depends upon dimension $d$.
\end{lemma}

\begin{lemma}\label{lemma: quotient}
For $g \in \mathrm H^1(\Omega) \cap \mathrm L^{\infty}(\Omega)$ satisfying the following bound
\[
0 < a \leq g(x) \leq b \qquad \text{ for a.e. } x \in \Omega
\]
we get $\frac{1}{g} \in \mathrm H^1(\Omega) \cap \mathrm L^{\infty}(\Omega) $ and 
\[
\nabla \left[\frac{1}{g}\right] = - \frac{1}{g^2} \nabla g
\]
\end{lemma}

\bibliographystyle{plain}
\begin{singlespace}
\bibliography{Quant-Para}
\end{singlespace}

\end{document}